\documentclass[a4paper,12pt]{amsart}

\usepackage{tikz} 
\usepackage{pgfplots}
\usetikzlibrary{angles,quotes}

\usepackage{amssymb}
\usepackage{amsmath}
\usepackage{amsthm}
\usepackage{mathtools}

\usepackage{enumerate}

\usepackage{latexsym}
\usepackage{mathrsfs}

\usepackage[colorlinks]{hyperref}
\usepackage{hyperref}
\renewcommand\eqref[1]{(\ref{#1})} 

\allowdisplaybreaks \numberwithin{equation}{section}

\theoremstyle{plain}
\newtheorem{theorem}{Theorem}[section]

\newtheorem{corollary}[theorem]{Corollary}
\newtheorem{lemma}[theorem]{Lemma}

\theoremstyle{definition}
\newtheorem{definition}[theorem]{Definition}
\newtheorem{remark}[theorem]{Remark}

\pgfplotsset{compat=1.18}
\begin{document}

\title[Almost sectorial operators]{Abstract Volterra integral equations of wave type with almost sectorial operators}

\author[J.E. Restrepo]{Joel E. Restrepo}
\address{\textcolor[rgb]{0.00,0.00,0.84}{Joel E. Restrepo \newline Department of Mathematics \newline Center for Research and Advanced Studies, Av. IPN 2508, 07360, CDMX, Mexico}}
\email{\textcolor[rgb]{0.00,0.00,0.84}{joel.restrepo@cinvestav.mx}}
	\endgraf

\date{\today}

\begin{abstract} 
We study classical solutions (existence, uniqueness, and explicit solution operator) for homogeneous, linear, and semilinear abstract Volterra integral equations of wave type with almost sectorial operators. We use a functional calculus for the latter type of operators to construct a general class of bounded linear operators that in particular contains the propagators (solution operators) of the considered equations. Some properties of this family of operators are also given.   
\end{abstract}



\maketitle

\tableofcontents

\section{Introduction}

From the 1960s the study of the following abstract Volterra integral equation has been popular: 
\begin{equation}\label{intro-main-e}
 w(t)+\int_0^t k(t-s)Aw(s){\rm d}s=f(t),\quad 0<t<T\leqslant +\infty,   
\end{equation}
where $A$ can be any linear or nonlinear operator (mainly unbounded) in $X$ (a complex Banach space), $k$ is a scalar kernel $\neq0$ and $f$ is a function in a suitable space (e.g., in $L^1_{loc}(\mathbb{R}^+,X)$). Here, we recommend checking \cite{clement,fri,book-new,pruss} and the references therein. There are different ways to study the equation \eqref{intro-main-e}. One of them is to consider the differential counter part as we explain in the next paragraphs. First, note that in applications, the kernel $g_\alpha(t)=\frac{t^{\alpha-1}}{\Gamma(\alpha)}$ for $0<\alpha<2$ plays an important role; see, e.g. \cite{Mainardi,pruss} for more details. Suppose $f\equiv0,$ thus the equation \eqref{intro-main-e} with kernel $g_\alpha$ can be rewritten as
\begin{equation}\label{intro-p-kernel}
w(t)+\int_0^{t}\frac{(t-s)^{\alpha-1}}{\Gamma(\alpha)}Aw(s){\rm d}s=w(t)+\prescript{RL}{0}I_t^{\alpha} Aw(t)=0,
\end{equation}
where $\prescript{RL}{0}I^{\rho}u(t)=\frac1{\Gamma(\rho)}\int_0^t (t-s)^{\rho-1}u(s)\,\mathrm{d}s$ is the Riemann-Liouville fractional integral of order $\rho>0.$ For $0<\alpha<1,$ and at least assuming that $w\in L^1((0,T);X),$ we can apply the well-known Djrbashian-Caputo fractional derivative  
\[
\prescript{C}{}\partial_{t}^{\alpha}w(t)=\prescript{RL}{0}I^{1-\alpha}\partial_t w(t)=\int_0^t \frac{(t-s)^{-\alpha}}{\Gamma(1-\alpha)}\partial_{s}w(s){\rm d}s,\]
to the above equation and by \cite[Formula (1.21)]{thesis2001} get that
\begin{equation}\label{intro-f-heat-e}
\prescript{C}{}\partial_{t}^{\alpha}w(t)+Aw(t)=0,    
\end{equation}
which is the nonlocal in time abstract heat type equation. Note that for $\alpha=1,$ we obtain the classical partial derivative in time, i.e. $^{C}\partial_{t}^{\alpha}w(t)=\prescript{RL}{0}I^{0}\partial_t w(t)=\partial_t w(t)$ since $\prescript{RL}{0}I^{0}$ acts as an identity operator. For the frame $0<\alpha<1,$ linear and semilinear nonlocal in time heat type equations have been considered by using different types of operators $A,$ for instance, sectorial or almost sectorial; see, e.g. \cite{tmn,uno2,{section3}}. In addition, the classical case of an abstract heat equation is well known \cite{Base,JEE2002}. For the latter types of operators, the solution operator of \eqref{intro-f-heat-e} can be written as follows:
\begin{equation}\label{heatfa}
E_\alpha(-t^{\alpha}A)=\int_0^{+\infty}M_{\alpha}(s)e^{-st^{\alpha}A}{\rm d}s,\quad t\geqslant0,    
\end{equation}
where $\{e^{-tA}\}_{t\geqslant0}$ is the $C_0-$ semigroup generated by $-A,$ and
\begin{equation}\label{wright}
M_{\alpha}(z)=\sum_{n=0}^{+\infty}\frac{(-z)^n}{n!\Gamma(-\alpha n+1-\alpha)},\quad z\in\mathbb{C},\quad 0\leqslant\alpha<1,
\end{equation}
is a Wright-type function that is convergent in the whole complex plane \cite{1940}. The expression \eqref{heatfa} can also be seen; for example, in \cite[Theorem 3.1]{thesis2001} or \cite[Theorem 2.42]{thesis}. Some of the basic properties of the function \eqref{wright} are the following:
\[
M_\alpha(t)\geqslant 0\quad\text{for any}\quad t\in(0,+\infty),\quad \int_0^{+\infty}M_\alpha(s){\rm d}s=1.
\]
The expression \eqref{heatfa} is clearly the most used in many works because the estimates depend on the $C_0-$ semigroup (see the references above and also \cite{navier,chen}) and those are usually well known and sharp. In fact, these propagators can be treated easily, since we arrive at the end to analyze a real-value integral. 

In contrast to the case of the nonlocal heat-type equation, the wave-type  solution operator is more delicate, i.e. $1<\alpha<2$. Indeed, a well-behaved propagator as the one presented in \eqref{heatfa} cannot be expected. The wave-type propagator will have involved an oscillatory principal term given by the two parametric Mittag-Leffler function (see Theorems \ref{bounded} and \ref{h-cs-we}), which is totally different from the behavior of \eqref{heatfa} where the function $E_{\alpha}(-t^{\alpha}x)$ $(t,x\geqslant 0)$ is completely monotonic \cite{Pollard}. Hence, different analysis and estimations compared with the heat case need to be developed in this scenario. Also, for $1<\alpha<2,$ the Djrbashian-Caputo operator becomes
\[
\prescript{C}{}\partial_{t}^{\alpha}w(t)=\prescript{RL}{0}I^{2-\alpha}\partial_t^{(2)} w(t)=\int_0^t \frac{(t-s)^{1-\alpha}}{\Gamma(2-\alpha)}\partial_{s}^{(2)}w(s){\rm d}s.\]
By similar arguments as those given for the heat case, equation \eqref{intro-p-kernel} turns out to be
\begin{equation}\label{intro-f-wave-e}
\prescript{C}{}\partial_{t}^{\alpha}w(t)+Aw(t)=0,    
\end{equation}
which is a homogeneous nonlocal in time abstract wave-type equation. The case $\alpha=2,$ is not considered since it goes back to the classical abstract wave equation, and it requests other techniques and methods for its study. 

\medskip In this paper, we focus mainly on studying linear and semilinear wave type equations with almost sectorial operators and nonlocal operators in time. Note that from 2012, there have been no advances with respect to the linear and semilinear problem \eqref{intro-f-wave-e}, i.e. the Volterra equation of wave type. Here, we solve this question in detail. The most recent study in this direction was given in \cite{2025}. Previous works in this direction can be found in \cite{Base,2025,JEE2002,section3}. The classical heat equation was considered in \cite{Base,JEE2002}, while the nonlocal in time abstract heat type equation was studied in \cite{section3}. Moreover, in \cite{2025}, the linear case of nonlocal in time abstract wave type equations was analyzed on some H\"older spaces. In the latter paper, wave-type propagators in different functional spaces were also studied, in particular, in some H\"older ones. Our study is based on the functional calculus developed for almost sectorial operators \cite{JEE2002}, which is an alternative and different way to use it in the analysis and construction of solution operators of several types of equations. In addition, we study the classical solutions of the equations considered. In the following, more details will be given.    

\medskip Different functional calculi have been constructed for several types of operators. Usually, the operator's spectrum lies in a region of the complex plane whose resolvent satisfies certain bounds \cite{{f-calculi-baty},haase,pruss}. These abstract calculi are very useful in the study of linear and semilinear partial integro-differential equations \cite{more-intro,lunardi,JEE2002}. One of the most extensively studied are the so-called sectorial operators. These types of operators are usually closed, linear, and densely defined, and the resolvent satisfies the estimate $\|(z-A)^{-1}\|\leqslant C|z|^{-1},$ for any $z$ in a suitable domain that does not contain the spectrum of the operator. Some important elliptic differential operators are in the class of sectorial operators. The generic spaces for these operators are, for example, Lebesgue spaces and continuous functions. However, operators defined in more regular spaces, for example in H\"older continuous functions, are not sectorial \cite{wahl}. The latter operators belong to the class of almost sectorial operators (see, e.g. \cite{haase,JEE2002,r-18}), which are closed linear operators $A:D(A)\subset X\to X$ defined in a complex Banach space $(X,\|\cdot\|)$ whose domain $D(A)$ is a linear subspace of $X,$ such that the spectrum $\sigma(A)$ is contained in the sector $S_\omega:=\{z\in\mathbb{C}\setminus\{0\}: |\text{arg}\,z|\leqslant\omega\}\cup \{0\}$ for some $0\leqslant \omega<\pi,$ i.e.

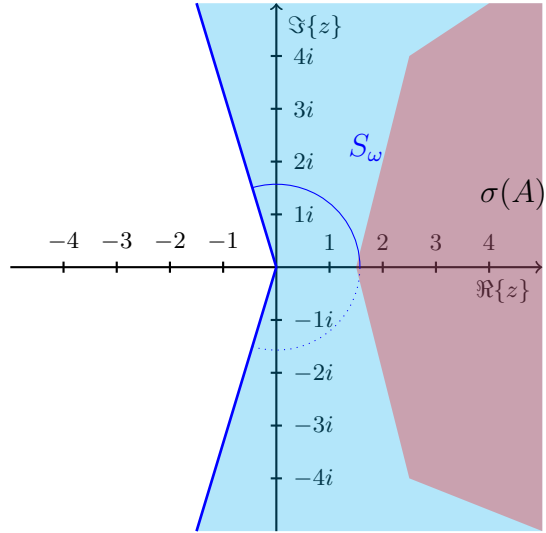
\begin{center}
\usetikzlibrary {angles,quotes}
\begin{tikzpicture}[scale=0.7]
    \begin{scope}[thick,font=\scriptsize]
    \draw [->] (-5,0) -- (5,0) node [below left] {$\Re\{z\}$};
    \draw [->] (0,-5) -- (0,5) node [below right] {$\Im\{z\}$};
    \iffalse

\draw [line width=0.5pt] (-2.5,5) -- (0,0) node ;
    
    \draw (1,-3pt) -- (1,3pt)   node [above] {$1$};
    \draw (-1,-3pt) -- (-1,3pt) node [above] {$-1$};
    \draw (-3pt,1) -- (3pt,1)   node [right] {$i$};
    \draw (-3pt,-1) -- (3pt,-1) node [right] {$-i$};
    \else
    \foreach \n in {-4,...,-1,1,2,...,4}{%
        \draw (\n,-3pt) -- (\n,3pt)   node [above] {$\n$};
        \draw (-3pt,\n) -- (3pt,\n)   node [right] {$\n i$};
    }
    \end{scope}

\fill[line width=3.pt,color=cyan,fill=cyan,fill opacity=0.3] (-1.5,5) -- (5.,5.) -- (5,-5.) -- (-1.5,-5) -- (0.,0.) -- cycle;
    
\fill[line width=3.pt,color=red,fill=red,fill opacity=0.3] (2.5,4) -- (4,5) -- (5,5) -- (5,-5) -- (2.5,-4) -- (1.5,0) -- cycle;

\draw [line width=1pt, color=blue] (0,0)-- (-1.5,5);
\draw [line width=1pt, color=blue] (0,0)-- (-1.5,-5);

\coordinate (O) at (0,0);
\coordinate (A) at (5,0);
\coordinate (B) at (-1.5,-5);
\coordinate (D) at (-1.5,5);
\coordinate (E) at (2,0);
\coordinate (F) at (5,5);

\pic [draw,blue, "$S_\omega$", angle eccentricity=1.8, angle radius=1.1cm] {angle = A--O--D};

\pic [draw,dotted, blue,-, angle eccentricity=2, angle radius=1.1cm] {angle = B--O--A};

\pic [-,black,"$\sigma(A)$", angle eccentricity=1.8, angle radius=1.1cm] {angle = A--E--F};

\end{tikzpicture}
\end{center}

and the resolvent satisfies the following estimate 
\begin{equation}\label{intro-res}
    \|(z-A)^{-1}\|\leqslant C_{\mu}|z|^{\gamma},\quad\text{for any}\quad z\notin S_\mu\quad\text{and}\quad \omega<\mu<\pi,\quad -1<\gamma<0,
    \end{equation}
where $C_{\mu}>0$ is a constant. In some cases, the researchers used a weaker condition than \eqref{intro-res}, which is given by:
\[
\|(z-A)^{-1}\|\leqslant C_{\mu}(1+|z|^{\upsilon})^{-1},\quad\text{for any}\quad z\notin S_\mu,\quad\upsilon\in(0,1), \quad\text{and}\quad \omega<\mu<\pi.
\]
More examples of these types of operators can be described by considering some special dumbbell domains, in particular, a dumbbell with a thin handle; see, e.g. \cite{Base}. 

By $\Theta_\omega^\gamma(X)$ we denote the set of all closed linear operators $A:D(A)\subset X\to X$ that are almost sectorial. For simplicity, we denote $\Theta_\omega^\gamma$ instead of $\Theta_\omega^\gamma(X)$. By notation, we write $S_\mu^0$ as the open sector $\{z\in\mathbb{C}\setminus\{0\}: |\text{arg}\,z|<\mu\}.$ Note that $0\in\rho(A)$ for any $A\in\Theta_\omega^\gamma.$ Also, we have that $A$ is injective \cite[Remark 2.2]{JEE2002}. It is important to recall that operators in the class $\Theta_\omega^\gamma$ have the possibility of having non-dense domain and/or range. This feature gives a different view with respect to the classical results, where dense domains are generally considered \cite{pazy,pruss}.

Although some researchers studied almost sectorial operators defined over domains that are dense as well, see, e.g. \cite{dense-1,dense-2}. In this paper, we consider the most general case. It is clear that with estimate \eqref{intro-res}, the operator $A$ cannot generate a $C_0-$ semigroup. In our case, almost sectorial operators generate another type of semigroups called {\it{analytic semigroups of growth order $\gamma$}}: 

\begin{definition}\label{analytic-s}
Let $0<\mu<\pi/2$ and $\kappa>0.$ A family $\{\mathscr{T}(t):t\in S_\mu^0\}$ is said to be an analytic semigroup of growth order $\kappa$ if the following conditions hold:
\begin{enumerate}
    \item $\mathscr{T}(t+s)=\mathscr{T}(t)\mathscr{T}(s)$ for any $t,s\in S_\mu^0.$
    \item \label{ii}The mapping $t\to\mathscr{T}(t)$ is analytic in $S_\mu^0.$
    \item There exists a positive constant $C$ such that 
    \[
    \|\mathscr{T}(t)\|\leqslant Ct^{-\kappa},\quad\text{for any}\quad t>0.
    \]
    \item If $\mathscr{T}(t)x=0$ for some $t\in S_\mu^0$, then $x=0.$
\end{enumerate}
\end{definition}
The above concept was almost the same as that introduced by Da Prato in \cite{r-2} for positive integer orders. There are just two differences with respect to the one given by Da Prato. First, the set $X_0=\bigcup_{t>0}\mathscr{T}(t)X$ not need to be dense in $X.$ Second, the strong continuity of the mapping $t\to \mathscr{T}(t)$ for $t>0$ is replaced by condition \eqref{ii}. The generalization for any positive order was given in different works e.g. \cite{r-13,oka,r-18,r-22} and the references therein. These semigroups do not imply strong continuity at $t=0,$ and this is one of the main difference compare with the $C_0$-semigroups. These type of operators frequently appear by the consideration of elliptic operators in regular spaces. Let us now recall briefly some classical examples in this setting. For example, $-\Delta_{\mathbb{R}^n}$ in a bounded domain $\Omega$ of $\mathbb{R}^n$ is sectorial under some suitable boundary conditions in $L^p(\Omega)$ \cite[Section 1.3]{6-intro}. Moreover, it is also sectorial in the spaces of bounded or continuous functions \cite{more-intro,19-intro}. While, in the space of H\"older continuous functions, $-\Delta_{\mathbb{R}^n}$ is almost sectorial, see e.g. \cite{wahl} or \cite[Example 3.1.33]{lunardi}. Some other good examples of almost sectorial operators can be found in \cite[Section 2]{JEE2002} or \cite{wahl}. 

\medskip Let us now give a brief summary of the main results of this manuscript.

\medskip From now on, we denote by $g_{\beta}(t)=\frac{t^{\beta-1}}{\Gamma(\beta)}$ for $\beta>0$ and $t>0.$ Also, $(v\ast u)(t)$ denotes the Laplace convolution, i.e.
$$(v\ast u)(t)=\int\limits_0^t v(t-s)u(s){\rm d}s.$$

In all the following statements, we assume that $A\in \Theta_\omega^\gamma,$ $-1<\gamma<0$ and $\omega<\theta<\mu<\pi-\alpha\frac{\pi}{2}.$ The restriction $\mu<\pi-\alpha\frac{\pi}{2}$ is explained in Theorem \ref{bounded}. Also, the domain $D(A)$ of an operator $A$ is always endowed under the graph norm $\|x\|_{D(A)}=\|Ax\|+\|x\|$, therefore, it is a Banach space. The positive constant $C$ that appears through the paper can vary from one step to another. 

\medskip First, in Section \ref{havewt}, we provide a general class of bounded linear operators related to the two-parametric Mittag-Leffler function $E_{\alpha,\delta}(z)$ $(\alpha<2,\delta\in\mathbb{R},z\in\mathbb{C})$ (see Subsection \ref{ML-s} for more details of these types of functions). Note that these operators will be involved in the representation of solutions for the linear and semilinear cases of our wave type equations. 

\begin{theorem}\label{bounded-intro}
 For any fixed $t>0$, the following operator 
\begin{equation}\label{repre-o-intro}
E_{\alpha,\delta}(-t^{\alpha}z)(A)=\frac{1}{2\pi i}\int_{\Gamma_\theta}E_{\alpha,\delta}(-t^{\alpha}z)(z-A)^{-1}{\rm d}z,\quad \delta\in\mathbb{R},\quad \alpha<2,
\end{equation}
is linear and bounded on $X,$ where $\|E_{\alpha,\delta}(-t^{\alpha}z)(A)\|\leqslant C(\alpha,\delta,\gamma)t^{-\alpha(1+\gamma)}$ for some positive constant $C(\alpha,\delta,\gamma).$ Also, the operator $E_{\alpha,\delta}(-t^{\alpha}z)(A)$ is strongly continuous in $[t_0,+\infty)$ for every $t_0>0$. In addition, for $x\in X,$ the mapping $t\to t^{\delta-1}E_{\alpha,\delta}(-t^{\alpha}z)(A)x$ is $n\,(\text{for any}\,\,n\in\mathbb{N})$ times continuously differentiable such that \[
\partial_t^{n}\big(t^{\delta-1}E_{\alpha,\delta}(-t^{\alpha}z)(A)\big)=t^{\delta-n-1}E_{\alpha,\delta-n}(-t^{\alpha}z)(A),
\]
and the latter operator defines a bounded linear operator in $X,$ where \[
\|\partial_t^{n}\big(t^{\delta-1}E_{\alpha,\delta}(-t^{\alpha}z)(A)\big)\|\leqslant C(\alpha,\delta-n,\gamma)t^{\delta-n-1-\alpha(1+\gamma)}.
\]
\end{theorem}
There are many other properties about these operators that are established in Section \ref{havewt}. We continue with the study of the classical solutions (see Definition \ref{cshc}) of the following semilinear abstract Volterra equations of wave type:
\begin{align}\label{we-intro}
\prescript{C}{}\partial_{t}^{\alpha}w(t)-Aw(t)&=f(t,u(t)),\quad 0<t\leqslant T,\quad 1<\alpha<2,\nonumber \\
w(t)|_{_{_{t=0}}}&=w_0, \\
\partial_t w(t)|_{_{_{t=0}}}&=w_1,\nonumber
\end{align}
where $X$ is a complex Banach space, $A\in \Theta_\omega^\gamma$ with $-1<\gamma<0$ and $\omega<\theta<\mu<\pi-\alpha\frac{\pi}{2}$ and $w_0,w_1\in D(A).$ We see some restrictions on the data and also on the sector that will be argued for the homogeneous, linear and semilinear cases of equation \eqref{we-intro} through different sections. For instance, for the homogeneous case, these questions and other remarks are treated fully in Section \ref{havewt}. In the following, we start by presenting the classical solution of the homogeneous problem \eqref{we-intro} $(f\equiv0).$

\begin{theorem}\label{h-cs-we-intro}
If $\frac{1}{1+\gamma}>\alpha>\frac{1}{-\gamma},$ $(-1<\gamma<-1/2)$ and $w_0,w_1\in D(A),$ then the classical solution of the homogeneous equation \eqref{we-intro} $(f\equiv0)$ is given by $w(t)=E_{\alpha}(-t^{\alpha}z)(A)w_0+tE_{\alpha,2}(-t^{\alpha}z)(A)w_1.$
\end{theorem}

Now we provide the main results for the classical solutions of the linear case of \eqref{we-intro}.
\begin{theorem}\label{thm-linear-we-intro}
Suppose that $f(t)\in D(A)$ for any $t\in(0,T]$, $f\in L^1\big((0,T);D(A)\big)$ and let $f$ be H\"older continuous with an exponent $\nu\in(0,1]$ such that $\nu>\alpha(1+\gamma)$, i.e., $\|f(t)-f(s)\|\leqslant \kappa |t-s|^{\nu},$ for $0<t,s\leqslant T.$ Then $w(t)=(g_{\alpha-1}(s)\ast E_{\alpha}(-s^{\alpha}A)\ast f(s))(t)$ is the unique classical solution of \eqref{we-intro} with $w_0=w_1=0.$
\end{theorem}
We can illustrate it as follows:
    
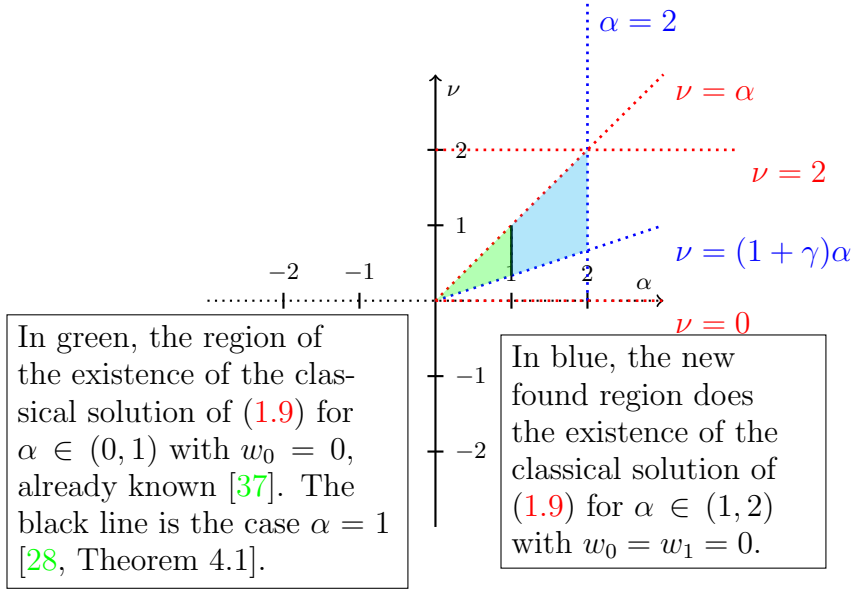
\begin{center}
\begin{tikzpicture}
    \begin{scope}[thick,font=\scriptsize]
    
        \draw [dotted,->] (-3,0) -- (3,0) node [above left] {$\alpha$};
    \draw [->] (0,-3) -- (0,3) node [below right] {$\nu$};
    
     \iffalse
    \else
    
    \foreach \n in {-2,-1,1,2}{%
        \draw (\n,-3pt) -- (\n,3pt)   node [above] {$\n$};
        \draw (-3pt,\n) -- (3pt,\n)   node [right] {$\n$};
    }
    \end{scope}

\draw [dotted, line width=1pt, color=blue] (0,0)-- (3,1) node [below right] {$\nu=(1+\gamma)\alpha$};

\draw [dotted, line width=1pt, color=blue] (2,0)-- (2,4) node [below right] {$\alpha=2$};

\draw [dotted, line width=1pt, color=red] (0,2)-- (4,2) node [below right] {$\nu=2$};


\draw [dotted, line width=1pt, color=black] (1,0.5)-- (1,1);

\draw [dotted, line width=1pt, color=red] (0,0)-- (3,3) node [below right] {$\nu=\alpha$};

\draw [dotted, line width=1pt, color=red] (0,0)-- (3,0) node [below right] {$\nu=0$};

\draw [line width=1pt, color=black] (1,1/3)-- (1,1);

\fill[line width=3.pt,color=green,fill=green,fill opacity=0.3] (1,1/3)--(1,1)--(0,0)--(1,1/3) -- cycle;


\fill[line width=3.pt,color=cyan,fill=cyan,fill opacity=0.3] (2,2)--(2,2/3)--(1,1/3)--(1,1)--(2,2) -- cycle;

\node[draw,text width=4cm] at (3,-2) {In blue, the new found region does the existence of the classical solution of \eqref{we-intro} for $\alpha\in(1,2)$ with $w_0=w_1=0$.};

\node[draw,text width=5cm] at (-3,-2) {In green, the region of the existence of the classical solution of \eqref{we-intro} for $\alpha\in(0,1)$ with $w_0=0,$ already known \cite{section3}. The black line is the case $\alpha=1$ \cite[Theorem 4.1]{JEE2002}.};

\end{tikzpicture}
\end{center}

\begin{theorem}\label{general-intro}
Assume that $\frac{1}{1+\gamma}>\alpha>\frac{1}{-\gamma},$ $f(t)\in D(A)$ for any $t\in(0,T]$, $f\in L^1\big((0,T);D(A)\big)$ and let $f$ be H\"older continuous with an exponent $\nu\in(0,1]$ such that $\nu>\alpha(1+\gamma).$ If $w_0,w_1\in D(A)$, then 
\[
w(t)=E_{\alpha}(-t^{\alpha}z)(A)w_0+tE_{\alpha,2}(-t^{\alpha}z)(A)w_1+(g_{\alpha-1}(s)\ast E_{\alpha}(-s^{\alpha}A)\ast f(s))(t)
\]
is the unique classical solution of \eqref{we-intro}. 
\end{theorem}

    Let us give a graphic of the region where Theorem \ref{general-intro} holds:

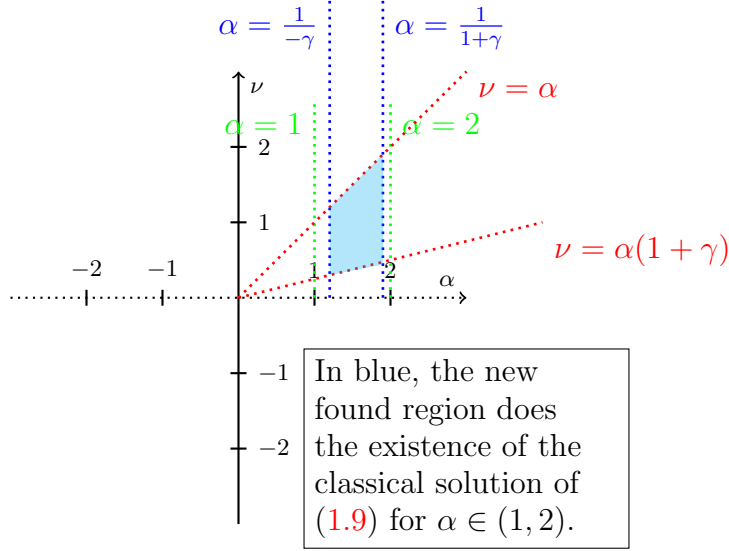
\begin{center}
\begin{tikzpicture}
    \begin{scope}[thick,font=\scriptsize]
    \draw [dotted,->] (-3,0) -- (3,0) node [above left] {$\alpha$};
    \draw [->] (0,-3) -- (0,3) node [below right] {$\nu$};
    \iffalse

    \else
    \foreach \n in {-2,-1,1,2}{%
        \draw (\n,-3pt) -- (\n,3pt)   node [above] {$\n$};
        \draw (-3pt,\n) -- (3pt,\n)   node [right] {$\n$};
    }
    \end{scope}

\draw [dotted, line width=1pt, color=red] (0,0)-- (4,1) node [below right] {$\nu=\alpha(1+\gamma)$};


\draw [dotted, line width=1pt, color=blue] (1.9,0)-- (1.9,4) node [below right] {$\alpha=\frac{1}{1+\gamma}$};

\draw [dotted, line width=1pt, color=green] (2,0)-- (2,2.6) node [below right] {$\alpha=2$};

\draw [dotted, line width=1pt, color=green] (1,0)-- (1,2.6) node [below left] {$\alpha=1$};

\draw [dotted, line width=1pt, color=blue] (1.2,0)-- (1.2,4) node [below left] {$\alpha=\frac{1}{-\gamma}$};




\draw [dotted, line width=1pt, color=red] (0,0)-- (3,3) node [below right] {$\nu=\alpha$};;



\fill[line width=3.pt,color=cyan,fill=cyan,fill opacity=0.3] (1.9,0.5)--(1.9,1.9)--(1.2,1.2)--(1.2,0.3)--(1.9,0.44)-- (1.9,1.9) -- cycle;

\node[draw,text width=4cm] at (3,-2) {In blue, the new found region does the existence of the classical solution of \eqref{we-intro} for $\alpha\in(1,2).$};

\end{tikzpicture}
\end{center}
In the last Section \ref{semilinear-s}, we analyze the classical solution of the semilinear equation \eqref{we-intro}. In this case, we first show the existence and uniqueness of a mild solution in $C([0,T];D(A))$ (see Definition \ref{mild-s}) for the equation \eqref{we-intro}. The results are read as follows:
\begin{theorem}
Suppose that the nonlinear function $f(t,x):[0,T]\times X\to D(A)$ is continuous with respect to the time variable $t$ such that 
\begin{equation*}
    \|f(t,x)-f(t,y)\|_{D(A}\leqslant L\|x-y\|\quad\text{for any}\quad t\in[0,T]\quad\text{and}\quad x,y\in X,
\end{equation*}
for some constant $L>0.$ If $w_0,w_1,Aw_0\in D(A)$ and $1>\alpha(1+\gamma)$ then the problem \eqref{we-intro} has a unique mild solution in $C([0,T];D(A)).$ 
\end{theorem}

Finally, we conclude this section with the statement on the classical solutions of \eqref{we-intro}. There is a delicate step in defining the domain and range of the nonlinear term that will be discussed with more arguments in Section \ref{semilinear-s}. 

\begin{theorem}
Suppose that for any $k>0,$ there exits a constant $L(k)$ such that the function $f:[0,T]\times X\to D(A)$ satisfies 
\begin{equation*}
        \|f(t,w)-f(s,v)\|_{D(A)}\leqslant L(k)\big(|t-s|^{\nu}+\|w-v\|\big),\quad\text{for some}\quad \nu>\alpha(1+\gamma),
    \end{equation*}
    for any $t,s\in [0,T],$ $w,v\in X$ with $\|w\|,\|v\|\leqslant k.$ If $\frac{1}{1+\gamma}>\alpha>\frac{1}{-\gamma},$ $w_0,w_1\in D(A)$ and $w\in C([0,T];X)$ is a mild solution of \eqref{we-intro}, then $w$ is a classical solution of \eqref{we-intro}.
\end{theorem}

\section{Preliminaries}\label{preli}

In this section, we recall and collect several results and concepts that will be used in the entire paper. First, we begin by recalling a well-known special function (the two-parametric Mittag-Leffler function) and their properties. This function will play an important role in the development of this paper. After that, we continue with the definition of several classes of holomorphic functions that will help us to introduce some results about the functional calculus involving the almost sectorial operators. In the end, we recall some definitions and results on nonlocal operators in time.   

\subsection{Mittag-Leffler function}\label{ML-s}

Frequently, we will use the two-parametric Mittag-Leffler function:
\begin{equation}\label{bimittag}
E_{\alpha,\delta}(z)=\sum_{k=0}^{+\infty} \frac{z^k}{\Gamma(\alpha k+\delta)},\quad z,\delta\in\mathbb{C},\quad \Re(\alpha)>0,
\end{equation}
which is an entire function, absolutely and locally uniformly convergent for the given parameters (\cite{mmdjr,mittag}). Usually, if $\delta=1,$ we denote $E_{\alpha,1}(z)$ simply by $E_{\alpha}(z).$ Some classical examples are: $E_{0,\delta}(z)=\frac{1}{\Gamma(\delta)}\frac{1}{1-z},$ $E_{1}(z)=\exp(z),$ $E_{\frac {1}{2}}(z)=\exp(z^{2})\operatorname {erfc}(-z),$ $E_{2}(z)=\cosh({\sqrt {z}}),$ $E_{2}(-z^{2})=\cos(z),$ $E_{1,2}(z)={\frac {e^{z}-1}{z}},$ $E_{2,2}(z)={\frac {\sinh({\sqrt {z}})}{\sqrt {z}}},$ etc. More examples can be found in \cite{mmdjr,mittag}. To estimate the propagators associated with these Mittag-Leffler functions, we recall the inequality \cite[Theorem 1.6]{page 35}: 
\begin{equation}\label{uniform-estimate}
|E_{\alpha,\delta}(z)|\leqslant \frac{C}{1+|z|},\quad z\in\mathbb{C},\quad \delta\in\mathbb{R},\quad\alpha<2,
\end{equation}
where $\mu\leqslant |\arg(z)|\leqslant \pi$, $\pi\alpha/2<\mu<\min\{\pi,\pi \alpha\}$ and $C$ is a positive constant. The region of the complex numbers that satisfies the estimate \eqref{uniform-estimate} is given in the following picture:
\begin{center}
\begin{tikzpicture}[scale=0.7]
    \begin{scope}[thick,font=\scriptsize]
    \draw [->] (-5,0) -- (5,0) node [below left] {$\Re\{z\}$};
    \draw [->] (0,-5) -- (0,5) node [below right] {$\Im\{z\}$};
    \iffalse

\draw [line width=0.5pt] (-2.5,5) -- (0,0) node ;
    
    \draw (1,-3pt) -- (1,3pt)   node [above] {$1$};
    \draw (-1,-3pt) -- (-1,3pt) node [above] {$-1$};
    \draw (-3pt,1) -- (3pt,1)   node [right] {$i$};
    \draw (-3pt,-1) -- (3pt,-1) node [right] {$-i$};
    \else
    \foreach \n in {-4,...,-1,1,2,...,4}{%
        \draw (\n,-3pt) -- (\n,3pt)   node [above] {$\n$};
        \draw (-3pt,\n) -- (3pt,\n)   node [right] {$\n i$};
    }
    \end{scope}

\fill[line width=3.pt,color=cyan,fill=cyan,fill opacity=0.3] (-3.5,5) -- (-5.,5.) -- (-5,-5.) -- (-3.5,-5) -- (0.,0.) -- cycle;
    
\fill[line width=3.pt,color=red,fill=red,fill opacity=0.3] (-2.5,5) -- (-5.,5.) -- (-5,-5.) -- (-2.5,-5) -- (0.,0.) -- cycle; 
\draw [line width=0.5pt, color=red] (0,0)-- (-2.5,5);
\draw [dotted, line width=1pt, color=blue] (0,0)-- (-3.5,5);
\draw [dotted, line width=1pt, color=blue] (0,0)-- (-3.5,-5);
\draw [line width=0.5pt, color=red] (0,0)-- (-2.5,-5);    

\coordinate (M) at (0,0);
\coordinate (N) at (5,0);
\coordinate (R) at (-2.5,5);
\coordinate (S) at (-3.5,5);

\pic [draw,blue, angle eccentricity=1.2, angle radius=1.3cm, "$\mu$"] {angle = N--M--S};

\pic [draw,red, angle eccentricity=1.5, angle radius=0.5cm,"$\alpha\pi/2$"] 
{angle = N--M--R};
\end{tikzpicture}
\end{center}
Also, by \cite[Theorem 1.4]{page 35}, we have
\begin{equation}\label{uniform-estimate-2}
|E_{\alpha,\alpha-n}(z)|\leqslant \frac{C}{1+|z|^2}\quad\text{as}\quad |z|\to+\infty,\quad\alpha<2,\quad \quad n=0,1,2,3,\ldots.
\end{equation}


\subsection{Almost sectorial operators (the class $\Theta_\omega^\gamma$)}
Let $0<\mu<\pi$. So, we recall that $S_\mu^0$ is the open sector $\{z\in\mathbb{C}\setminus\{0\}: |\text{arg}\,z|<\mu\},$ and its closure $S_\mu:=\{z\in\mathbb{C}\setminus\{0\}: |\text{arg}\,z|\leqslant\mu\}\cup \{0\}.$ We consider the function $\text{arg}$ with values in $(-\pi,\pi].$ Set 
\[
\mathcal{F}_0^{\gamma}(S_\mu^0)=\bigcup_{s<0}\Psi_s^{\gamma}(S_\mu^0)\cup\Psi_0(S_\mu^0),
\]
and 
\[
\mathcal{F}(S_\mu^0)=\{f\in\mathcal{H}(S_\mu^0):\,\,\text{there exit}\,\,k,n\in\mathbb{N}\,\, \text{such that}\,\, f\psi_n^k\in \mathcal{F}_0^{\gamma}(S_\mu^0)\},
\]
where 
\[
\mathcal{H}(S_\mu^0)=\{f:S_\mu^0\to\mathbb{C};\,\,f\,\,\text{is holomorphic}\},
\]
\[
\mathcal{H}^{\infty}(S_\mu^0)=\{f\in \mathcal{H}(S_\mu^0),\,\,f\,\,\text{is bounded}\},
\]
\[
\varphi_0(z)=\frac{1}{1+z},\quad \psi_n(z)=\frac{z}{(1+z)^n},\quad z\in\mathbb{C}\setminus\{-1\},\,\,n\in\mathbb{N}\cup\{0\},
\]
\[
\Psi_0(S_\mu^0)=\left\{f\in\mathcal{H}(S_\mu^0):\sup_{z\in S_\mu^0}\left|\frac{f(z)}{\varphi_0(z)}\right|<+\infty\right\},
\]
and for each $s<0,$
\[
\Psi_s^{\gamma}(S_\mu^0)=\left\{f\in\mathcal{H}(S_\mu^0):\sup_{z\in S_\mu^0}|\psi_n^s(z)f(z)|<+\infty\right\},
\]
where $n$ is the smallest integer such that $n\geqslant2$ and $\gamma+1<-(n+1)s.$ Note that
\[
\mathcal{F}_0^\gamma(S_\mu^0)\subset \mathcal{H}^{\infty}(S_\mu^0)\subset \mathcal{F}(S_\mu^0)\subset \mathcal{H}(S_\mu^0), 
\]
and for $k,n\in\mathbb{N}\cup\{0\}$ with $n>k$, one has $\psi_n^k\in \mathcal{F}_0^{\gamma}(S_\mu^0).$

\subsection{Functional calculus of the class $\Theta_\omega^\gamma$} We recall some useful results on the functional calculus involving the almost sectorial operators.  

Below we always denote by $\Gamma_\theta$ $(0<\theta<\pi)$ the path 
\begin{equation}\label{path}
\{re^{-i\theta}:r>0\}\cup\{re^{i\theta}:r>0\}
\end{equation}
oriented such that the sector $S_\theta^0$ lies to the left of $\Gamma_\theta.$

\begin{theorem}\label{thm-main}
Let $A\in \Theta_\omega^\gamma$ and $\omega<\theta<\mu<\pi.$ The following statements hold:
\begin{enumerate}
    \item For every $f\in \mathcal{F}_0^\gamma(S_\mu^0)$, the integral 
\begin{equation}\label{integral-fc}
f(A)=\frac{1}{2\pi i}\int_{\Gamma_\theta}f(z)(z-A)^{-1}{\rm d}z
\end{equation}
is absolutely convergent and defines a bounded linear operator on $X.$ Also, its value does not depend on the choice of $\theta$ for $\omega<\theta<\mu.$
\item (Product formula) For all $f,g\in \mathcal{F}_0^\gamma(S_\mu^0)$, we have $fg\in \mathcal{F}_0^\gamma(S_\mu^0)$ and $(fg)(A)=f(A)g(A).$

Moreover, for all $f,g\in \mathcal{F}(S_\mu^0)$, we have that $f(A)g(A)\subset(fg)(A).$ Also, if $D[(fg)(A)]\subset D[g(A)],$ then $f(A)g(A)=(fg)(A).$ Furthermore, if $g(A)$ is bounded, then $f(A)g(A)=(fg)(A).$ 
\end{enumerate}
 
\end{theorem}

\subsection{Nonlocal operators in time}
We begin by recalling some Sobolev spaces. Hence, let $I=(0,T)$ for some $T>0,$ $n\in\mathbb{N},$ $1\leqslant q<+\infty$ and:
\[
W^{n,q}(I;X):=\left\{u\,\,\big/\,\, \exists \phi\in L^{q}(I;X):\,\, u(t)=\sum_{j=0}^{n-1}a_j \frac{t^{j}}{j!}+\frac{t^{n-1}}{(n-1)!}\ast \phi(t),\,\, t\in I\right\}.
\]
Here, $\phi(t)=u^{(n)}(t),$ $a_j=u^{(j)}(0)$ and $X$ is a complex Banach space.

\medskip For $u\in L^1(I;X),$ we recall the Riemann-Liouville fractional integral of order $\rho>0$:
\[
\prescript{RL}{0}{I}^{\rho}u(t)=\frac1{\Gamma(\rho)}\int_0^t (t-s)^{\rho-1}u(s)\,\mathrm{d}s=(g_{\rho}\ast u)(t),\quad t>0,
\]
where $\prescript{RL}{0}{I}^{0}u(t)=u(t).$

Below, $\lceil \rho \rceil$ denotes the smallest integer greater than or equal to $\rho.$

\begin{definition}
Let $u\in L^1(I;X)$ and $g_{n-\rho}\ast u\in W^{n,1}(I;X)$ $(\lceil \rho \rceil=n).$ The Riemann-Liouville fractional derivative of order $\rho$ is defined as
\[
\prescript{RL}{0}D^{\rho}u(t)=\left(\frac{{\rm d}}{{\rm d}t}\right)^n \prescript{RL}{0}I^{n-\rho}u(t).
\]
\end{definition}
\begin{definition}
Assume that $u\in W^{n,1}(I;X),$ then the Djrbashian-Caputo fractional derivative of order $\rho>0$ is defined by 
\[
\prescript{C}{0}D^{\rho}u(t)=\prescript{RL}{0}{I}^{n-\rho}\left(\frac{{\rm d}}{{\rm d}t}\right)^n u(t).
\]
Also, note that for $u\in C^{n-1}(I;X),$ $g_{n-\rho}\ast u\in W^{n,1}(I;X)$ $(n\in\mathbb{N},0\leqslant n-1<\rho<n),$ we get 
\[
\prescript{C}{0}D^{\rho}u(t)=\left(\frac{{\rm d}}{{\rm d}t}\right)^{n}\prescript{RL}{0}{I}^{n-\rho}\left(u(t)-\sum_{k=0}^{n-1}u^{(k)}(0)g_{k+1}(t)\right).
\]
\end{definition}
The above operator is usually called the regularized Caputo (or Djrbashian-Caputo) fractional derivative. In this paper, we always use this regularized operator.  

\begin{remark}
Let us provide some important details in the real case about the difference to use the Djrbashian-Caputo fractional derivative and its regularized version. Thus, we fix a finite interval $[a,T]\subseteq\mathbb{R}.$ We denote by:
\begin{align*}
AC[a,T]&=\left\{f:[a,T]\to\mathbb{R}\;:\;f\text{ absolutely continuous on }[a,T]\right\}; \\
AC^n[a,T]&=\left\{f:[a,T]\to\mathbb{R}\;:\;f^{(n-1)}\text{ exists and in }AC[a,T]\right\},\qquad n\in\mathbb{N}.
\end{align*}
For any function $f\in AC^n[a,T]$, both definitions are equivalent \cite[Theorem 2.2]{samko}. However, the main difference between the two definitions is the possibility of defining the regularized version in a larger function space. An illustrative example of this situation is given in \cite{example}, where a function has not a first-order derivative but has Riemann--Liouville fractional derivative of all orders less than one.
\end{remark}

\section{Homogeneous abstract Volterra equations of wave type}\label{havewt}

In this section, we consider two kind of homogeneous abstract Volterra equations of wave type. The first case is when the operator $A$ belongs to the class $\Theta_\omega^\gamma.$ While, the second one deals with a power of the operator $A.$ 

\subsection{$A$-Abstract Volterra equations of wave type}
We study the following equation:
\begin{align}\label{we}
\prescript{C}{}\partial_{t}^{\alpha}w(t)-Aw(t)&=0,\quad 0<t\leqslant T,\quad 1<\alpha<2,\nonumber \\
w(t)|_{_{_{t=0}}}&=w_0, \\
\partial_t w(t)|_{_{_{t=0}}}&=w_1,\nonumber
\end{align}
where $X$ is a complex Banach space, $A\in \Theta_\omega^\gamma$ with $\omega<\theta<\mu<\pi-\alpha\frac{\pi}{2}$ and in principle (this will be clarified later) $w_0,w_1\in X.$

\medskip Let us recall a definition of the nature of the solutions of \eqref{we}. 

\begin{definition}\label{cshc}
A function $w\in C([0,T];X)$ is called a classical solution of the problem \eqref{we} if $^C\partial_t^{\alpha}w\in C((0,T];X),$ $w(t)\in D(A)$ for all $t\in(0,T]$ and satisfies the problem \eqref{we}.
\end{definition}

\begin{remark}
An interesting issue in the above definition is the fact that $w\in C([0,T];X)$ but $^C\partial_t^{\alpha}w\in C((0,T];X).$ Therefore, we provide an illustrative example in the real numbers $\mathbb{R}$ of it. In fact, take $w(t)=t^{\beta}$ for $2>\alpha>\beta>1$ and $t\in [0,1].$ Clearly, $w\in C([0,1];\mathbb{R}).$ Also, by performing some elementary calculations, we obtain $\prescript{C}{}\partial_t^{\alpha}w(t)=\frac{\Gamma(2-\alpha)\Gamma(\beta-1)}{\Gamma(2-\alpha+\beta-1)}t^{-\alpha+\beta}.$ Since $-\alpha+\beta<0$ we have an integrable singularity at $t=0$ of the latter function, and hence $\prescript{C}{}\partial_t^{\alpha}w\in C((0,1];\mathbb{R}).$
\end{remark}

The solution of equation \eqref{we} is connected with the following operators:
\[
E_\alpha(-t^{\alpha}z)(A)=\frac{1}{2\pi i}\int_{\Gamma_\theta}E_{\alpha}(-t^{\alpha}z)(z-A)^{-1}{\rm d}z,\quad t>0,
\]
and
\[
tE_{\alpha,2}(-t^{\alpha}z)(A)=\frac{t}{2\pi i}\int_{\Gamma_\theta}E_{\alpha,2}(-t^{\alpha}z)(z-A)^{-1}{\rm d}z,\quad t>0,
\]
where $\Gamma_\theta$ is the integral contour of \eqref{path}, $A\in \Theta_\omega^\gamma,$ $-1<\gamma<0$ and $\omega<\theta<\mu<\pi-\alpha\frac{\pi}{2}.$ Usually, in some contexts, we also denote $E_\alpha(-t^{\alpha}z)(A)$ simply by $E_\alpha(-t^{\alpha}A).$ 

\begin{remark}\label{rerewrite}
    From the equality $tE_{\alpha,2}(-t^{\alpha}z)=\big(1\ast E_{\alpha}(-s^{\alpha}z)\big)(t)$ for any $1<\alpha<2$ and $z\in\mathbb{C},$ we can rewrite the operator $tE_{\alpha,2}(-t^{\alpha}A)$ as $tE_{\alpha,2}(-t^{\alpha}A)=(1\ast E_{\alpha}(-s^{\alpha}A))(t).$
\end{remark}

Below we provide some properties of a more general class of operators, in particular, it contains $E_\alpha(-t^{\alpha}A)$ and $tE_{\alpha,2}(-t^{\alpha}A)$.

\begin{theorem}\label{bounded}
Let $A\in \Theta_\omega^\gamma,$ $-1<\gamma<0$ and $\omega<\theta<\mu<\pi-\alpha\frac{\pi}{2}.$ For each fixed $t>0$, the following operator 
\begin{equation}\label{repre-o}
E_{\alpha,\delta}(-t^{\alpha}z)(A)=\frac{1}{2\pi i}\int_{\Gamma_\theta}E_{\alpha,\delta}(-t^{\alpha}z)(z-A)^{-1}{\rm d}z,\quad \delta\in\mathbb{R},\quad \alpha<2,
\end{equation}
is linear and bounded on $X.$ Besides, the operator $E_{\alpha,\delta}(-t^{\alpha}z)(A)$ is strongly continuous in $[t_0,+\infty)$ for every $t_0>0$. Moreover, there exists a positive constant $C(\alpha,\delta,\gamma)$ such that
\begin{equation}\label{asymp-g}
\|E_{\alpha,\delta}(-t^{\alpha}z)(A)\|\leqslant C(\alpha,\delta,\gamma)t^{-\alpha(1+\gamma)}.
    \end{equation}
Also, the mapping $t\to t^{\delta-1}E_{\alpha,\delta}(-t^{\alpha}z)(A)$ is $n\,(\text{for any}\,\,n\in\mathbb{N})$ times continuously differentiable such that 
\[
\partial_t^{n}\big(t^{\delta-1}E_{\alpha,\delta}(-t^{\alpha}z)(A)\big)=t^{\delta-n-1}E_{\alpha,\delta-n}(-t^{\alpha}z)(A),
\]
and it defines a bounded linear operator in $X.$ Here we have 
\[
\|\partial_t^{n}\big(t^{\delta-1}E_{\alpha,\delta}(-t^{\alpha}z)(A)\big)\|\leqslant C(\alpha,\delta-n,\gamma)t^{\delta-n-1-\alpha(1+\gamma)}.
\]
\end{theorem}
\begin{proof}
We begin by showing that the function $E_{\alpha,\delta}(-t^{\alpha}z)$ belongs to $\mathcal{F}_0^\gamma(S_\mu^0)$ for some fixed $t>0$ and any point $z\in S_\mu^0.$ First, we assume that $0\leqslant\text{arg}(z)<\mu.$ Therefore, $\text{arg}(-t^{\alpha}z)=-\pi+\text{arg}(z).$ Thus, $|\text{arg}(-t^{\alpha}z)|=\pi-\text{arg}(z)>\pi-\mu>\alpha \pi/2$ since $\mu<\pi-\alpha\frac{\pi}{2}.$ This means that we can use the estimate \eqref{uniform-estimate} to obtain $|E_{\alpha,\delta}(-t^{\alpha}z)|\leqslant \frac{C}{1+t^{\alpha}|z|}\leqslant \frac{C_t}{1+|z|}.$ Analogously, we have the same latter estimates for the case $0\geqslant\text{arg}(z)>-\mu.$ Hence the function $E_{\alpha,\delta}(-t^{\alpha}z)$ is in the class $\mathcal{F}_0^\gamma(S_\mu^0).$ Note that by Theorem \ref{thm-main}, the operator $E_{\alpha,\delta}(-t^{\alpha}z)(A)$ is a well-defined bounded linear operator on $X,$ and it has the representation \eqref{repre-o}. Also, let $\omega<\tilde{\mu}<\theta.$ Then, for any $z\in\Gamma_\theta,$ we have
\begin{align}
\|E_{\alpha,\delta}(-t^{\alpha}z)(A)\|&\leqslant \frac{1}{2\pi}\int_{\Gamma_\theta}|E_{\alpha,\delta}(-t^{\alpha}z)|\|(z-A)^{-1}\||{\rm d}z| \nonumber\\
&\leqslant \frac{C C_{\tilde{\mu}}}{\pi}\int_0^{+\infty}\frac{r^{\gamma}}{1+t^{\alpha}r}{\rm d}r=\frac{C(\alpha,\delta,\gamma)}{t^{\alpha \gamma+\alpha}},\label{u-esti}  
\end{align}
hence it follows \eqref{asymp-g}. Let us now fix $t_0>0.$ So, for any $t>0$ and $x\in X$, we have that 
\[
E_{\alpha,\delta}(-t^{\alpha}z)(A)x-E_{\alpha,\delta}(-t_0^{\alpha}z)(A)x=\frac{1}{2\pi i}\int_{\Gamma_\theta}\big(E_{\alpha,\delta}(-t^{\alpha}z)-E_{\alpha,\delta}(-t_0^{\alpha}z)\big)(z-A)^{-1}x{\rm d}z.
\]
By the estimate \eqref{u-esti} and the Dominated Convergence Theorem, we obtain that
\[
\lim_{t\to t_0^{+}}E_{\alpha,\delta}(-t^{\alpha}z)(A)x=E_{\alpha,\delta}(-t_0^{\alpha}z)(A)x,
\]
so that it is strongly continuous in $[t_0,+\infty).$

On the other hand, we already know that $E_{\alpha,\delta-n}(-t^{\alpha}z)(A)$ is a bounded linear operator in $X$ for any $n\in\mathbb{N}.$ By using this, the identity $\partial_t^{n}\big(t^{\delta-1}E_{\alpha,\delta}(-t^{\alpha}z)\big)=t^{\delta-n-1}E_{\alpha,\delta-n}(-t^{\alpha}z)$ which holds for any $n\in\mathbb{N}$ and $z\in\mathbb{C}$ \cite[Formula (4.9.5)]{mittag}, and by the application of $n$ times of the Leibniz integral rule, we have that 
\begin{align*}
t^{\delta-n-1}E_{\alpha,\delta-n}(-t^{\alpha}z)(A)x&=\frac{t^{\delta-n-1}}{2\pi i}\int_{\Gamma_\theta}E_{\alpha,\delta-n}(-t^{\alpha}z)(z-A)^{-1}x{\rm d}z,\quad x\in X, \\
&=\frac{1}{2\pi i}\int_{\Gamma_\theta}\partial_t^{n}\big(t^{\delta-1}E_{\alpha,\delta}(-t^{\alpha}z)\big)(z-A)^{-1}{\rm d}z \\
&=\left(\frac{\partial}{\partial t}\right)^n\left(\frac{t^{\delta-1}}{2\pi i}\int_{\Gamma_\theta}E_{\alpha,\delta}(-t^{\alpha}z)(z-A)^{-1}{\rm d}z\right) \\
&=\partial_t^{(n)}\big(t^{\delta-1}E_{\alpha,\delta}(-t^{\alpha}z)(A)\big),
\end{align*}
which means that the function  $t^{\delta-1}E_{\alpha,\delta}(-t^{\alpha}z)(A)$ is $n$ times continuously differentiable on $(0,+\infty).$ In addition, since the estimate \eqref{asymp-g} holds, we get that 
\begin{align*}
\|\partial_t^{n}\big(t^{\delta-1}E_{\alpha,\delta}(-t^{\alpha}z)(A)\big)\|&\leqslant t^{\delta-n-1}\|E_{\alpha,\delta-n}(-t^{\alpha}z)(A)\| \\
&\leqslant C(\alpha,\delta-n,\gamma)t^{\delta-n-1-\alpha(1+\gamma)}.
\end{align*} \end{proof}

\begin{remark}
It is known that the Gamma function $\Gamma(z)$ $(z\in\mathbb{C})$ has no zeros in the real line and poles at $z=0,-1,-2,\ldots$ \cite[Appendix A]{mittag}. Therefore, for some values of $\alpha,\delta,n,$ the function $E_{\alpha,\delta-n}(-t^{\alpha}z)$ in Theorem \ref{bounded} could be zero or has some terms equal to zero. In any case, this function must be carefully analyzed.  
\end{remark}
As an immediate consequence of Theorem \ref{bounded}, we obtain the following assertion.
\begin{corollary}
Let $A\in \Theta_\omega^\gamma,$ $-1<\gamma<0$ and $\omega<\theta<\mu<\pi-\alpha\frac{\pi}{2}.$ For each fixed $t>0$, $S_\alpha(t)$ and $T_\alpha(t)$ are $n\,(n\in\mathbb{N})$ times continuously differentiable on $(0,+\infty)$ and bounded linear operators on $X.$ Also, there exist positive constants $C_{1,2}(\alpha,\gamma)$ such that
\begin{equation}\label{asymp}
\|E_\alpha(-t^{\alpha}A)\|\leqslant C_{1}(\alpha,\gamma)t^{-\alpha(1+\gamma)},\quad \|tE_{\alpha,2}(-t^{\alpha}A)\|\leqslant C_{2}(\alpha,\gamma)t^{1-\alpha(1+\gamma)}.
    \end{equation}
\end{corollary}
In general, we cannot expect to have a representation like the one given in \eqref{repre-o} (see Theorem \ref{bounded}) when we consider the product operator $AE_{\alpha,\delta}(-t^{\alpha}A),$ but the next result gives a partial answer to it. 

\begin{theorem}\label{lem-a1}
Let $A\in \Theta_\omega^\gamma$ with $\omega<\theta<\mu<\pi-\alpha\frac{\pi}{2}$. Then the operator $(zE_{\alpha,\alpha-n}(-t^{\alpha}z))(A)$ is linear and bounded in $X$ for any $n=0,1,2,3,\ldots$ and each fixed $t>0.$ This operator can be represented by 
\begin{equation}\label{potra}
(zE_{\alpha,\alpha-n}(-t^{\alpha}z))(A)=\int_{\Gamma_\theta}zE_{\alpha,\alpha-n}(-t^{\alpha}z)(z-A)^{-1}{\rm d}z,    
\end{equation}
such that  
\[
\|(zE_{\alpha,\alpha-n}(-t^{\alpha}z))(A)\|\leqslant Ct^{-2\alpha-\alpha\gamma}.
\]
Moreover, the above operator is strongly continuous in $[t_0,+\infty),$ for every $t_0>0.$
\end{theorem}
\begin{proof}
The function $zE_{\alpha,\alpha-n}(-t^{\alpha}z)\in \mathcal{H}^{\infty}(S_\mu^0)$ due to the estimate \eqref{uniform-estimate}. Also, by the estimate \eqref{uniform-estimate-2}, for any $re^{i\theta}$ with $-\mu<\theta<\mu$ and $r>0,$ we have 
\[
\left|re^{i\theta}E_{\alpha,\alpha-n}(-t^{\alpha}re^{i\theta})\right|\leqslant C\frac{r}{1+t^{2\alpha}r^2}\to  0,\quad\text{as}\quad r\to+\infty.
\]
We now prove that the function $z\to zE_{\alpha,\alpha-n}(-t^{\alpha}z)(z-A)^{-1}$ is absolutely integrable on $\Gamma_\theta$. In fact, let $\omega<\tilde{\mu}<\theta,$ and hence 
\[
\|zE_{\alpha,\alpha-n}(-t^{\alpha}z)(z-A)^{-1}\|\leqslant\left\{
\begin{array}{cl}
C|z|^{1+\gamma}(1+t^{\alpha}|z|)^{-1}&\mbox{if}\quad z\in\Gamma_\theta,\\
C|z|^{1+\gamma}(1+t^{2\alpha}|z|^{2})^{-1}&\mbox{if}\quad z\in\Gamma_\theta\,\,\text{with}\,\, |z|\to+\infty.
\end{array}\right. 
\]
Thus, for some $N_0$ large enough, we get
\begin{align*}
&\left\|\int_{\Gamma_\theta}zE_{\alpha,\alpha-n}(-t^{\alpha}z)(z-A)^{-1}{\rm d}z\right\| \\
&\hspace{1cm}\leqslant C\left(\int_0^{N_0}\frac{r^{1+\gamma}}{1+t^{\alpha}r}{\rm d}r+\int_{N_0}^{+\infty}\frac{r^{1+\gamma}}{1+t^{2\alpha}r^2}{\rm d}r\right)<+\infty,    
\end{align*}
which implies the absolutely integrability of the function on $\Gamma_\theta.$ Therefore, by \cite[Lemma 2.13]{JEE2002}, the operator $(zE_{\alpha,\alpha-n}(-t^{\alpha}z))(A)$ can be represented by the integral formula \eqref{potra} and it defines a bounded linear operator in $X.$ 

Note also that for $N_0=1/t^{\alpha}$ with some $t>0$ sufficient small, we get
\begin{align*}
\int_0^{1/t^{\alpha}}\frac{r^{1+\gamma}}{1+t^{\alpha}r}{\rm d}r+\int_{1/t^{\alpha}}^{+\infty}\frac{r^{1+\gamma}}{1+t^{2\alpha}r^2}{\rm d}r\leqslant C\frac{1}{t^{2\alpha+\alpha\gamma}}.    
\end{align*} 
And the strongly continuity in $[t_0,+\infty)$ follows by the latter estimate and the Dominated Convergence Theorem. 
\end{proof}

\begin{remark}\label{da}
   By Theorems \ref{thm-main} (item 2), \ref{bounded} and \ref{lem-a1}, we see that for any $x\in X$ and each fixed $t>0,$ $E_{\alpha,\alpha-n}(-t^{\alpha}z)(A)x$ belongs to the domain of $A.$ It is also clear that $(zE_{\alpha,\alpha-n}(-t^{\alpha}z))(A)=A(E_{\alpha,\alpha-n}(-t^{\alpha}z))(A).$
\end{remark}

Now we need to establish some elementary results that will help us to analyze the type of solution of our equations. Next, we give an alternative representation of the solution operator $E_\alpha(-t^{\alpha}A).$ The new representation will give the possibility to obtain some specific (special) properties that the representation in Theorem \ref{bounded} could not show at first glance. Before we provide some preliminary lemmas.

\medskip Below, we always consider the following Hankel path for $\theta_0\in\big(\pi/2,\frac{\pi-\theta}{\alpha}\big)$ and $\rho>0,$ where $\theta<\pi-\alpha\frac{\pi}{2}$:
\begin{align*}
\Gamma_{\theta_0}&:=\Gamma_{\theta_0}^1 \cup \Gamma_{\theta_0}^2 \cup \Gamma_{\theta_0}^3 \\
&\,\,=\{re^{-i\theta_0},\rho\leqslant r<+\infty\}\cup\{\rho e^{i\phi},-\theta_0\leqslant\phi<\theta_0\} \cup \{re^{i\theta_0},\rho\leqslant r<+\infty\}.
\end{align*}
Usually, for the next results, we assume $\rho=1/t$ for some $t>0.$

\begin{lemma}\label{lemma1}
Take $A\in \Theta_\omega^\gamma$ with $\omega<\theta<\mu<\pi-\alpha\frac{\pi}{2}$. The next representation is true 
\[
\left(\frac{1}{\lambda^{\alpha}+z}\right)(A)=\frac{1}{2\pi i}\int_{\Gamma_\theta}\frac{1}{\lambda^{\alpha}+z}(z-A)^{-1}{\rm d}z,\quad \lambda\in\Gamma_{\theta_0},  \]  
and defines a bounded linear operator in $X$ such that 
\[
\|(\lambda^{\alpha}+A)^{-1}\|\leqslant C|\lambda|^{\alpha\gamma}, \quad \lambda\in\Gamma_{\theta_0}.
\]
\end{lemma}
\begin{proof}
Let $\lambda\in\Gamma_{\theta_0}.$ We begin with $\lambda\in \Gamma_{\theta_0}^1.$ Then
\[
|\lambda^{\alpha}+z|\geqslant\cos\left(\frac{\theta+\alpha\theta_0}{2}\right)(\rho^{\alpha}+|z|)\geqslant C\cos\left(\frac{\theta+\alpha\theta_0}{2}\right)(1+|z|)>0
\]
since $\frac{\theta+\alpha\theta_0}{2}<\frac{\pi}{2}.$ Also, for $\lambda\in\Gamma_{\theta_0}^3$, we have
\[
|\lambda^{\alpha}+z|\geqslant\cos\left(\frac{\theta-\alpha\theta_0}{2}\right)(\rho^{\alpha}+|z|)\geqslant C\cos\left(\frac{\theta-\alpha\theta_0}{2}\right)(1+|z|)>0
\]
due to $-\frac{\pi}{2}<\frac{\theta-\alpha\theta_0}{2}<\frac{(1-\alpha)\pi}{2}.$ 

On the other hand, for $\lambda\in\Gamma_{\theta_0}^2$, in a similar way, we get
\[
|\lambda^{\alpha}+z|\geqslant C(1+|z|).
\]
Thus
\[
\frac{1}{|\lambda^{\alpha}+z|}\leqslant \frac{C}{1+|z|},\quad\text{for each}\quad\lambda\in\Gamma_{\theta_0}.
\]
Hence $\frac{1}{\lambda^{\alpha}+z}\in \mathcal{F}_0^{\gamma}(S_\mu^0)$ for any $\lambda\in\Gamma_{\theta_0},$ and the result follows from Theorem \ref{thm-main}.

Also, for $\omega<\tilde{\mu}<\theta$, we obtain
\begin{align*}
 \int_{\Gamma_\theta}\frac{1}{|\lambda^{\alpha}+z|}&\|(z-A)^{-1}\||{\rm d}z|\leqslant C\int_{\Gamma_\theta}\frac{1}{|\lambda|^{\alpha}+|z|}|z|^{\gamma}|{\rm d}z| \\
    &\leqslant C\int_0^{+\infty}\frac{r^{\gamma}}{|\lambda|^{\alpha}+r}{\rm d}r=C|\lambda|^{\gamma\alpha}\int_0^{+\infty}\frac{s^{\gamma}}{1+s}{\rm d}s<C|\lambda|^{\gamma}.
\end{align*}\end{proof}

Let us now present the alternative representation of the operator $E_{\alpha}(-t^{\alpha}A).$ In the next representation, we take $\rho=1/t$ for $t>0,$ on the path $\Gamma_{\theta_0}.$

\begin{theorem}\label{strong}
Let $A\in \Theta_\omega^\gamma,$ $-1<\gamma<0$ and $\omega<\theta<\mu<\pi-\alpha\frac{\pi}{2}.$ For any $t>0,$ it follows that 
\begin{equation}\label{alternative-r}
E_{\alpha}(-t^{\alpha}z)(A)=\frac{1}{2\pi i}\int_{\Gamma_{\theta_0}}e^{\lambda t}\lambda^{\alpha-1}(\lambda^{\alpha}+A)^{-1}{\rm d}\lambda,
\end{equation}
and it is strongly continuous for any $x\in D(A).$
\end{theorem}
\begin{proof}
We first show that the expression \eqref{alternative-r} converges. Indeed, by Lemma \ref{lemma1}, it yields 
\begin{align*}
    \frac{1}{2\pi}\int_{\Gamma_{\theta_0}}|e^{\lambda t}|&|\lambda^{\alpha-1}|\|(\lambda^{\alpha}+A)^{-1}\||{\rm d}\lambda|\leqslant C\int_{\Gamma_{\theta_0}}|e^{\lambda t}||\lambda|^{\alpha+\alpha\gamma-1}|{\rm d}\lambda| \\
    &\leqslant C\left(\int_{\rho}^{+\infty}e^{tr\cos(\theta_0)}r^{\alpha+\alpha\gamma-1}{\rm d}r+\rho^{\alpha+\alpha\gamma}\int_{-\theta_0}^{\theta_0}e^{t\rho\cos(\phi)}{\rm d}\phi\right) \\
    &\leqslant C\left(\frac{1}{t^{\alpha+\alpha\gamma}}\int_0^{+\infty}s^{\alpha(1+\gamma)-1}e^{-s}{\rm d}s+\rho^{\alpha+\alpha\gamma}\int_{-\theta_0}^{\theta_0}e^{t\rho\cos(\phi)}{\rm d}\phi\right)<+\infty
\end{align*}
for any $t>0.$ So, by Lemma \ref{lemma1} and Fubini's theorem, we have that 
\begin{align*}
\frac{1}{2\pi i}\int_{\Gamma_{\theta_0}}e^{\lambda t}\lambda^{\alpha-1}&(\lambda^{\alpha}+A)^{-1}{\rm d}\lambda \\
&=\frac{1}{2\pi i}\int_{\Gamma_{\theta_0}}e^{\lambda t}\lambda^{\alpha-1}\left(\frac{1}{2\pi i}\int_{\Gamma_\theta}\frac{1}{\lambda^{\alpha}+z}(z-A)^{-1}{\rm d}z\right){\rm d}\lambda \\   
&=\frac{1}{2\pi i}\int_{\Gamma_\theta}\left(\frac{1}{2\pi i}\int_{\Gamma_{\theta_0}}e^{\lambda t}\lambda^{\alpha-1}(\lambda^{\alpha}+z)^{-1}{\rm d}\lambda\right)(z-A)^{-1}{\rm d}z \\
&=\frac{1}{2\pi i}\int_{\Gamma_\theta}E_{\alpha}(-t^{\alpha}z)(z-A)^{-1}{\rm d}z,
\end{align*}
where the last representation is true since \cite[Theorem 1.1]{page 35} and $\pi-\theta>\frac{\pi-\theta}{\alpha}>\theta_0$.

Let us now see that the operator \eqref{alternative-r} is strongly continuous for any $x\in D(A).$ Since $(\lambda^{\alpha}+A)(\lambda^{\alpha}+A)^{-1}=I$ for any $\lambda\in \Gamma_{\theta_0},$ note that
\begin{align*}
E_{\alpha}(-t^{\alpha}z)(A)x-Ix&=\frac{1}{2\pi i}\int_{\Gamma_{\theta_0}}e^{\lambda t}\lambda^{\alpha-1}(\lambda^{\alpha}+A)^{-1}x{\rm d}\lambda-\frac{1}{2\pi i}\int_{\Gamma_{\theta_0}}e^{\lambda t}\lambda^{-1}{\rm d}\lambda Ix \\
&=\frac{1}{2\pi i}\int_{\Gamma_{\theta_0}}e^{\lambda t}\lambda^{-1}\left(\lambda^{\alpha}(\lambda^{\alpha}+A)^{-1}-I\right)x\,{\rm d}\lambda \\
&=-\frac{1}{2\pi i}\int_{\Gamma_{\theta_0}}e^{\lambda t}\lambda^{-1}(\lambda^{\alpha}+A)^{-1}Ax\,{\rm d}\lambda.
\end{align*}
Hence
\begin{align*}
    \|E_{\alpha}(-t^{\alpha}z)(A)x-Ix\|\leqslant C\|Ax\|\left(\int_{\rho}^{+\infty}e^{tr\cos(\theta_0)}r^{\alpha\gamma-1}{\rm d}r+\rho^{\alpha\gamma}\int_{-\theta_0}^{\theta_0}e^{t\rho\cos(\phi)}{\rm d}\phi\right).
\end{align*}
Since $\rho=1/t,$ it follows that
\begin{align*}
    \|E_{\alpha}(-t^{\alpha}z)(A)x-Ix\|\leqslant C\|Ax\|\left(t^{-\alpha\gamma}\int_{-\cos(\theta_0)}^{+\infty}e^{-s}s^{\alpha\gamma-1}{\rm d}s+t^{-\alpha\gamma}\int_{-\theta_0}^{\theta_0}e^{\cos(\phi)}{\rm d}\phi\right).
\end{align*}
It yields
\[ \|E_{\alpha}(-t^{\alpha}z)(A)x-Ix\|\to0,\quad{\text as}\quad t\to0^+.\]
\end{proof}
In the following, we prove some useful results that will help us establish the regularity of the solutions for homogeneous, linear, and semilinear equations in the next subsections. 
\begin{lemma}\label{i} Let $x\in X.$ If $t>0,$ then $E_{\alpha}(-t^{\alpha}A)x\in D(A)$ and
   \[
   \|AE_{\alpha}(-t^{\alpha}A)\|\leqslant C(t^{-\alpha}+t^{-2\alpha-\alpha\gamma}).
   \]
Besides, for any $x\in D(A)$ and $t>0$, we have 
\[
\|AE_{\alpha}(-t^{\alpha}A)x\|\leqslant C\|Ax\|t^{-\alpha(1+\gamma)}.
\]
\end{lemma}
\begin{proof}
From $(\lambda^{\alpha}+A)(\lambda^{\alpha}+A)^{-1}=I$ for any $\lambda\in \Gamma_{\theta_0},$ and the alternative representation \eqref{alternative-r} of the propagator $E_{\alpha}(-t^{\alpha}A),$ we obtain
\begin{align*}
AE_{\alpha}(-t^{\alpha}A)&=\frac{1}{2\pi i}\int_{\Gamma_{\theta_0}}e^{\lambda t}\lambda^{\alpha-1}A(\lambda^{\alpha}+A)^{-1}{\rm d}\lambda \\
&=\frac{1}{2\pi i}\int_{\Gamma_{\theta_0}}e^{\lambda t}\lambda^{\alpha-1}I{\rm d}\lambda-\frac{1}{2\pi i}\int_{\Gamma_{\theta_0}}e^{\lambda t}\lambda^{\alpha-1}\lambda^{\alpha}(\lambda^{\alpha}+A)^{-1}{\rm d}\lambda \\
&=K_1+K_2.
\end{align*}
One can see that $\|K_1\|\leqslant Ct^{-\alpha};$ while, by Lemma \ref{lemma1}, $\|K_2\|\leqslant Ct^{-2\alpha-\alpha\gamma}.$ Thus, it completes the first proof. The second one follows in a similar way from \eqref{alternative-r} and Lemma \ref{lemma1}.    
\end{proof}

Now we give a remark on the Laplace transform of the operator $E_\alpha(-t^{\alpha}A)$ which will be useful to establish the uniqueness of our solutions in equation \eqref{we}.

\begin{lemma}\label{laplace}
Let $A\in \Theta_\omega^\gamma,$ $-1<\gamma<0$ and $\omega<\theta<\mu<\pi-\alpha\frac{\pi}{2}.$ If $1>\alpha(1+\gamma),$ then the Laplace transform of $S_\alpha(t)$ exists and 
\[
    \widehat{E_\alpha(-t^{\alpha}A)}(\lambda)=\frac{\lambda^{\alpha-1}}{\lambda^{\alpha}I+A},\quad \lambda^{\alpha}\in \rho(-A).
    \]
\end{lemma}
\begin{proof}
    From the estimate \eqref{asymp}, one has that 
    \begin{align*}
    \int_0^{+\infty}e^{-vt}\|E_\alpha(-t^{\alpha}A)\|{\rm d}t&\leqslant C\int_0^{+\infty}e^{-vt}t^{-\alpha(1+\gamma)}{\rm d}t \\
    &=\frac{C}{v^{1-\alpha(1+\gamma)}}\int_0^{+\infty}e^{-u}u^{-\alpha(1+\gamma)}{\rm d}u<+\infty,
    \end{align*}
    whenever $1>\alpha(1+\gamma)$ for any $v>0.$ Thus, the Laplace transform of $E_\alpha(-t^{\alpha}A)$ exists, i.e. 
    \[
    \widehat{E_\alpha(-t^{\alpha}A)}(\lambda)=\int_0^{+\infty}e^{-\lambda t}E_\alpha(-t^{\alpha}A){\rm d}t,\quad {\rm Re}\,(\lambda) >0.
    \]
    Moreover, from the relation $\widehat{E_{\alpha}(-t^{\alpha}z)}(\lambda)=\frac{\lambda^{\alpha-1}}{\lambda^{\alpha}+z}$ that holds for any $z\in\Gamma_\theta$ and ${\rm Re}\,(\lambda)>0,$ Theorem \ref{bounded}, Lemma \ref{lemma1}, Fubini's theorem, \cite[Corollary B.3]{vv} and the analytic continuation, we have 
    \[
    \widehat{E_\alpha(-t^{\alpha}A)}(\lambda)=\frac{\lambda^{\alpha-1}}{2\pi i}\int_{\Gamma_\theta}\frac{1}{\lambda^{\alpha}+z}(z-A)^{-1}{\rm d}z=\frac{\lambda^{\alpha-1}}{\lambda^{\alpha}I+A},\quad \lambda^{\alpha}\in \rho(-A).
    \]\end{proof}

Below, we establish a result about the time-derivative of the operator $E_\alpha(-t^{\alpha}A).$

\begin{theorem}\label{derivative-s}
  Let $A\in \Theta_\omega^\gamma,$ $-1<\gamma<0$ and $\omega<\theta<\mu<\pi-\alpha\frac{\pi}{2}.$ For each $x\in X,$ we have
    \[
\partial_t E_\alpha(-t^{\alpha}A)x=-A(g_{\alpha-1}(s)\ast E_\alpha(-s^{\alpha}A))(t)x.
    \]
    Also, for each $x\in D(A)$ and $1>\alpha(1+\gamma),$
\[
    \partial_t E_\alpha(-t^{\alpha}A)x=-(g_{\alpha-1}(s)\ast AE_\alpha(-s^{\alpha}A))(t)x.
    \]
Moreover, if $\alpha>1/(-\gamma),$ that is $\frac{1}{1+\gamma}>\alpha>\frac{1}{-\gamma}$ $(-1<\gamma<-1/2)$, then \[
\lim_{t\to+0}\partial_t E_\alpha(-t^{\alpha}A)x=0,\quad\text{for all}\quad x\in D(A).
\]
\end{theorem}
\begin{proof}
    Note that by Theorems \ref{bounded} and \ref{lem-a1}, and the identities 
    \[
    \partial_t E_\alpha(-t^{\alpha}z)=-zt^{\alpha-1}E_{\alpha,\alpha}(-t^{\alpha}z),\quad t^{\alpha-1}E_{\alpha,\alpha}(-t^{\alpha}z)=\big(g_{\alpha-1}(s)\ast E_{\alpha}(-s^{\alpha}z)\big)(t),
    \]
for $z\in\mathbb{C}$ and $t>0,$ we obtain that 
    \[
 \partial_t E_\alpha(-t^{\alpha}A)=-A(g_{\alpha-1}(s)\ast E_\alpha(-s^{\alpha}A))(t).
    \]
    Now, by using \cite[Prop. 1.1.7]{vv} and Lemma \ref{i}, we have 
    \[
 \partial_t E_\alpha(-t^{\alpha}A)=(g_{\alpha-1}(s)\ast AE_\alpha(-s^{\alpha}A))(t),
    \]
    and for each $x\in D(A)$
 \begin{align*}
    \|\partial_t E_\alpha(-t^{\alpha}A)x\|&\leqslant C\|Ax\|\int_0^t (t-s)^{\alpha-2}s^{-\alpha(1+\gamma)}{\rm d}s \\
    &\leqslant C\|Ax\|t^{-1-\alpha\gamma}\to 0,\quad \text{if}\quad -\alpha\gamma>1.
    \end{align*}
    \end{proof}
    
\begin{remark}\label{conmu}
Another equivalent representation for the time derivative of $E_{\alpha}(-t^{\alpha}A)$ is given by $-t^{\alpha-1}(AE_{\alpha,\alpha}(-t^{\alpha}A))$. This result is followed by Theorem \ref{bounded} and the equality $\partial_t E_\alpha(-t^{\alpha}z)=-zt^{\alpha-1}E_{\alpha,\alpha}(-t^{\alpha}z).$
\end{remark}

The following result directly shows the classical solution of the problem \eqref{we}. This is the main result of this section.

\begin{theorem}\label{h-cs-we}
Let $t>0,$ $1<\alpha<2$ with $\frac{1}{1+\gamma}>\alpha>\frac{1}{-\gamma},$ $(-1<\gamma<-1/2).$ Let $\mu$ be such that $\mu<\pi-\alpha\frac{\pi}{2}.$ If $w_0,w_1\in D(A),$ then the classical solution of \eqref{we} is given by 
\[
w(t)=E_{\alpha}(-t^{\alpha}z)(A)w_0+tE_{\alpha,2}(-t^{\alpha}z)(A)w_1.
\]
\end{theorem}
\begin{proof}
We first check that $^{C}\partial_{t}^{\alpha}w(t)=Aw(t).$ We know that 
\[
^{C}\partial_{t}^{\alpha}\big(E_\alpha(-t^{\alpha}z)\big)=-zE_{
\alpha}(-t^{\alpha}z),\quad z\in\mathbb{C},\quad t>0,\quad 1<\alpha<2.
\]
It is clear that $z\in \mathcal{F}(S_\mu^0)$ and $E_{\alpha}(-t^{\alpha}z)\in \mathcal{F}_0^{\gamma}(S_\mu^0),$ hence $zE_{\alpha}(-t^{\alpha}z)\psi_2^1(z)\in \mathcal{F}_0^{\gamma}(S_\mu^0).$ By \cite[Def. 2.9]{JEE2002} and Theorem \ref{thm-main} (item 2) we can define the closed linear operator $\big[zE_{\alpha}(-t^{\alpha}z)\big](A)$ as follows 
$\big([\psi_2^1(z)](A)\big)^{-1}(zE_{\alpha}(-t^{\alpha}z)\psi_2^1(z))(A)$ where $(zE_{\alpha}(-t^{\alpha}z)\psi_2^1(z))(A)$ is a bounded linear operator in $X$ represented by the integral \eqref{integral-fc} and $\big([\psi_2^1(z)](A)\big)^{-1}$ is a closed linear operator in $X.$ Since $E_{\alpha}(-t^{\alpha}z)(A)$ is a bounded linear operator (by Theorem \ref{bounded}) we have that $\big[zE_{\alpha}(-t^{\alpha}z)\big](A)=A\big[E_{\alpha}(-t^{\alpha}z)\big](A)$ (Theorem \ref{thm-main} (item 2)). Thus
\begin{align*}
\left(\,^{C}\partial_{t}^{\alpha}E_\alpha(-t^{\alpha}z)(A)\right)=-\big(zE_\alpha(-t^{\alpha}z)\big)(A)=-A\big(E_\alpha(-t^{\alpha}z)\big)(A).
\end{align*}
Also, by \cite[Remark 2.6]{JEE2002}, the Leibniz integral rule, Fubini's theorem and Theorem \ref{thm-main} (item 2), we have for all $x\in X$ that
\begin{align*}
-AE_\alpha(-t^{\alpha}A)x&=\left(\,^{C}\partial_{t}^{\alpha}E_\alpha(-t^{\alpha}A)\right)x \\
&=-(zE_{\alpha}(-t^{\alpha}z)\psi_2^1(z))(A)\big([\psi_2^1(z)](A)\big)^{-1}x \\
&=(^{C}\partial_{t}^{\alpha}E_\alpha(-t^{\alpha}z)\psi_2^1(z))(A)\big([\psi_2^1(z)](A)\big)^{-1}x \\
&=\,^{C}\partial_{t}^{\alpha}\big(E_\alpha(-t^{\alpha}z)\psi_2^1(z))(A)\big)\big([\psi_2^1(z)](A)\big)^{-1}x \\
&=\,^{C}\partial_{t}^{\alpha}\big(E_\alpha(-t^{\alpha}A)\big)x.
\end{align*}
Since Theorems \ref{bounded} and \ref{strong}, it follows that $E_{\alpha}(-t^{\alpha}A)w_0\in C([0,T];X).$ Also, from the equality above, that is, $^{C}\partial_{t}^{\alpha}\left(E_\alpha(-t^{\alpha}z)(A)\right)x=-A\big(E_\alpha(-t^{\alpha}z)\big)(A)x$ for any $x\in X,$ and by Lemma \ref{i}, we have that $^{C}\partial_{t}^{\alpha}E_\alpha(-t^{\alpha}z)(A)x\in C\big((0,T];X\big).$ Moreover, by the same latter lemma, $E_{\alpha}(-t^{\alpha}A)x\in D(A)$ for all $t\in(0,T]$ and $x\in X.$   

On the other hand, by doing a similar reasoning as above, and from the equality 
\[
^{C}\partial_{t}^{\alpha}\big(tE_{\alpha,2}(-t^{\alpha}z)\big)=-ztE_{
\alpha,2}(-t^{\alpha}z),\quad z\in\mathbb{C},\quad t>0,\quad 1<\alpha<2,
\]
we obtain
\begin{align*}
\big(\prescript{C}{}\partial_{t}^{\alpha}[tE_{\alpha,2}(-t^{\alpha}z)]\big)(A)=-tA\big(E_{\alpha,2}(-t^{\alpha}z)\big)(A).
\end{align*}
From the estimate \eqref{u-esti} and the Dominated Convergence Theorem, we see that $tE_{\alpha,2}(-t^{\alpha}z)(A)x$ is continuous in $(0,T]$ for any $t>0$ and $x\in X.$ And, again by \eqref{u-esti} and the condition $1>\alpha(1+\gamma),$ we obtain the continuity of the propagator $tE_{\alpha,2}(-t^{\alpha}z)(A)x$ in $t=0,$ so, $tE_{\alpha,2}(-t^{\alpha}z)(A)x$ belongs to $C([0,T];X)$.   


We know that 
$\big(\prescript{C}{}\partial_{t}^{\alpha}[tE_{\alpha,2}(-t^{\alpha}z)]\big)(A)x=-tA\big(E_{\alpha,2}(-t^{\alpha}z)\big)(A)x,$ hence by Remark \ref{rerewrite}, Lemma \ref{i} and \cite[Lemma 1.1.7]{vv}, it follows that 
\begin{align*}
\|tA\big(E_{\alpha,2}(-t^{\alpha}z)\big)(A)w_1\|&=\|A\big(1\ast E_{\alpha}(s^{\alpha}A)(t)w_1\big)\| \\
&\leqslant \int_0^t \|AE_{\alpha}(-s^{\alpha}A)w_1\|{\rm d}s\leqslant C\|Aw_1\|\int_0^t s^{-\alpha(1+\gamma)}{\rm d}s \\
&\leqslant C\|Aw_1\|t^{1-\alpha(1+\gamma)}<+\infty. 
\end{align*}
So, $\big(\prescript{C}{}\partial_{t}^{\alpha}[tE_{\alpha,2}(-t^{\alpha}z)]\big)(A)w_1$ is bounded for any $t>0,$ and thus this function is in $C((0,T];X).$ It is also clear that $tE_{\alpha,2}(-t^{\alpha}z)(A)w_1\in D(A)$ for all $t\in(0,T].$ The uniqueness of the solution follows by the Laplace transform, Lemma \ref{laplace} and the uniqueness of the inverse Laplace transform. Note that by Theorem \ref{strong}, the estimate \eqref{u-esti} and $1>\alpha(1+\gamma),$ we get that $w(0)=w_1,$ while $w^{\prime}(t)=E_{\alpha}^{\prime}(-t^{\alpha}A)w_0+E_{\alpha}(-t^{\alpha}A)w_1$ implies that $w^{\prime}(0)=w_1$ due to Theorem \ref{derivative-s} and $\frac{1}{1+\gamma}>\alpha>\frac{1}{-\gamma}$. Finally, we have proven that $w(t)=E_{\alpha}(-t^{\alpha}z)(A)w_0+tE_{\alpha,2}(-t^{\alpha}z)(A)w_1$ is a classical solution of \eqref{we}, which completes the proof.
\end{proof}

\begin{remark}\label{remark-main-thm} Let us highlight the following aspects:
\begin{enumerate}
    \item In Theorem \ref{h-cs-we}, the first issue to mention is that if $w_0\in X$ and $w_1=0$ then we lost the continuity of the classical solution at $t=0,$ i.e. $E_{\alpha}(-t^{\alpha}z)(A)w_0\in C((0,T];X).$
    \item It is important to note that condition $1>\alpha(1+\gamma)$ in Theorem \ref{h-cs-we} is sufficient to guarantee the continuity of the propagator $tE_{\alpha,2}(-t^{\alpha}A)x$ at $t=0$ for any $x\in X.$ Moreover, by using both representations of $tE_{\alpha,2}(-t^{\alpha}A),$ that is, $tE_{\alpha,2}(-t^{\alpha}A)=(1\ast E_{\alpha}(-s^{\alpha}A))(t),$ such a restriction appeared. Nevertheless, it is not known whether this condition is necessary for it. This is an open question at this stage. With the current approach and methods, we do not see how this can be revealed. At this moment, to establish the classical solutions of \eqref{we} such a condition cannot be avoided. In the case $0<\alpha<1,$ this condition did not appear, but the initial data must be in the domain of $A,$ i.e. $u_0\in D(A)$, see \cite[Theorems 3.4 and 4.1]{section3}.
    \item Note that if $w_1=0$ in Theorem \ref{h-cs-we}, the additional condition $\alpha>\frac{1}{-\gamma}$ is not necessary to guarantee a classical solution of \eqref{we}.
\end{enumerate}    
\end{remark}

\subsection{$A^\beta$-Abstract Volterra equations of wave type} In this subsection, we consider an abstract wave equation of Volterra type by means of powers of the operator $A.$ In this setting, it is known that if $A\in\Theta_\omega^\gamma$ and $1+\gamma<\beta<\pi/\omega$ then $A^{\beta}\in\Theta_{\beta\omega}^{-1+\frac{\gamma+1}{\beta}}$ \cite[Prop. 3.6]{JEE2002}.

So, we study the following $A^{\beta}$-wave type equation:
\begin{align}\label{we-power}
^{C}\partial_{t}^{\alpha}w(t)-A^{\beta}w(t)&=0,\quad 0<t\leqslant T,\quad 1<\alpha<2,\nonumber \\
w(t)|_{_{_{t=0}}}&=w_0, \\
\partial_t w(t)|_{_{_{t=0}}}&=w_1,\nonumber
\end{align}
where $X$ is a complex Banach space, $A\in \Theta_\omega^\gamma$ with $(1+\gamma)\omega<\beta\omega<\theta<\mu<\pi-\alpha\frac{\pi}{2}$ and $w_0,w_1\in X.$ From Theorems \ref{bounded} and \ref{h-cs-we}, it is clear that the solution of equation \eqref{we-power} is related with the following operators:
\[
E_\alpha(-t^{\alpha}z)(A^{\beta})=\frac{1}{2\pi i}\int_{\Gamma_\theta}E_{\alpha}(-t^{\alpha}z)(z-A^{\beta})^{-1}{\rm d}z,\quad t>0,
\]
and
\[
tE_{\alpha,2}(-t^{\alpha}z)(A^{\beta})=\frac{t}{2\pi i}\int_{\Gamma_\theta}E_{\alpha,2}(-t^{\alpha}z)(z-A^{\beta})^{-1}{\rm d}z,\quad t>0,
\]
where $A\in \Theta_\omega^\gamma,$ $-1<\gamma<0$ and $(1+\gamma)\omega<\beta\omega<\theta<\mu<\pi-\alpha\frac{\pi}{2}.$ And, we also have:
\begin{corollary}
Let $A\in \Theta_\omega^\gamma,$ $-1<\gamma<0$ and $(1+\gamma)\omega<\beta\omega<\theta<\mu<\pi-\alpha\frac{\pi}{2}.$ For each fixed $t>0$, $E_\alpha(-t^{\alpha}A^{\beta})$ and $tE_{\alpha,2}(-t^{\alpha}A^{\beta})$ are $n\,(\text{for any}\,\,n\in\mathbb{N})$ times continuously differentiable and bounded linear operators on $X.$ Also, there exist positive constants $C_{1,2}(\alpha,\gamma)$ such that
\begin{equation}\label{asymp-}
\|E_\alpha(-t^{\alpha}A^{\beta})\|\leqslant C_1(\alpha,\gamma)t^{-\alpha(\gamma+1)/\beta},\quad \|tE_{\alpha,2}(-t^{\alpha}A^{\beta})\|\leqslant C_2(\alpha,\gamma)t^{-\frac{\alpha(1+\gamma)}{\beta}+1}.
    \end{equation}
\end{corollary}

\begin{corollary}
Let $t$ be any positive real number, $1<\alpha<2$ with $\frac{\beta}{1+\gamma}>\alpha>\frac{\beta}{\beta-(1+\gamma)}.$ Let $\mu$ be such that $\beta\omega<\theta<\mu<\pi-\alpha\frac{\pi}{2}.$ If $w_0,w_1\in D(A),$ then the classical solution of \eqref{we-power} is given by
\[
w(t)=E_{\alpha}(-t^{\alpha}A^\beta)w_0+tE_{\alpha,2}(-t^{\alpha}A^\beta)w_1.
\]    
\end{corollary}

\section{Linear abstract Volterra equations of wave type}

In this section, we consider the following abstract wave equation of Volterra type:
\begin{align}\label{we-f}
\prescript{C}{}\partial_{t}^{\alpha}w(t)-Aw(t)&=f(t),\quad 0<t\leqslant T,\quad 1<\alpha<2,\nonumber \\
w(t)|_{_{_{t=0}}}&=w_0, \\
\partial_t w(t)|_{_{_{t=0}}}&=w_1,\nonumber
\end{align}
where $X$ is a complex Banach space, $A\in \Theta_\omega^\gamma,$ $-1<\gamma<0,$ $\omega<\theta<\mu<\pi-\alpha\frac{\pi}{2}$, $w_0,w_1\in X$ and $f\in L^1\big((0,T);X\big).$

\medskip We first prove some preliminary results. 

\begin{lemma}\label{777}
    Let $A\in \Theta_\omega^\gamma,$ $-1<\gamma<0$ and $\omega<\theta<\mu<\pi-\alpha\frac{\pi}{2}.$ For $x\in X$ and $t>0,$ we have that $\big(g_{\alpha-1}(s)\ast E_{\alpha}(-s^{\alpha}A)\big)(t)x\in D(A)$ and 
    \[
    \|A\big(g_{\alpha-1}(s)\ast E_{\alpha}(-s^{\alpha}A)\big)(t)\|\leqslant Ct^{-1-\alpha(1+\gamma)}.
    \]
    Now, if $x\in D(A)$ and $1>\alpha(1+\gamma),$ it follows that $\big(g_{\alpha-1}(s)\ast E_{\alpha}(-s^{\alpha}A)\big)(t)x\in D(A)$ and 
    \[
    \|A\big(g_{\alpha-1}(s)\ast E_{\alpha}(-s^{\alpha}A)\big)(t)x\|\leqslant C\|Ax\|t^{-1-\alpha\gamma}.
    \]
\end{lemma}
\begin{proof}
By Theorem \ref{bounded} and \cite[Formula (1.100)]{page 35} we get
\begin{align*}
t^{\alpha-1}E_{\alpha,\alpha}(-t^{\alpha}z)(A)&=\frac{1}{2\pi i}\int_{\Gamma_\theta}\big(t^{\alpha-1}E_{\alpha,\alpha}(-t^{\alpha}z)\big)(z-A)^{-1}{\rm d}z \\
&=\frac{1}{2\pi i}\int_{\Gamma_\theta}\big(g_{\alpha-1}(s)\ast E_{\alpha}(-s^{\alpha}z)\big)(t)(z-A)^{-1}{\rm d}z \\
&=\big(g_{\alpha-1}(s)\ast E_{\alpha}(-s^{\alpha}A)\big)(t).
\end{align*}
Also, by Theorem \ref{lem-a1}, we obtain 
\[
\|A\big(g_{\alpha-1}(s)\ast E_{\alpha}(-s^{\alpha}A)\big)(t)\|=\|t^{\alpha-1}AE_{\alpha,\alpha}(-t^{\alpha}z)(A)\|\leqslant Ct^{-1-\alpha(1+\gamma)}.
\]
On the other hand, if $x\in D(A)$ and $1>\alpha(1+\gamma),$ from Lemma \ref{i} and \cite[Prop. 1.1.7]{vv}, we get 
\begin{align*}
\|A\big(g_{\alpha-1}(s)\ast E_{\alpha}(-s^{\alpha}A)\big)(t)x\|&\leqslant C\|Ax\|\big(g_{\alpha-1}(s)\ast s^{-\alpha(1+\gamma)}\big)(t) \\
&\leqslant C\|Ax\|t^{-\alpha\gamma-1}\int_0^1 (1-s)^{\alpha-2}s^{-\alpha(1+\gamma)}{\rm d}s,
\end{align*}
which is bounded whenever $1-\alpha(1+\gamma)>0$ and $t>0.$ 
\end{proof}

Let us now prove the following assertion.
\begin{lemma}\label{uno}
For each $x\in D(A),$ $t>0$ and $1>\alpha(1+\gamma),$ we have that $\big(g_{\alpha}(s)\ast E_{\alpha}(-s^{\alpha}A)\big)(t)x\in D(A)$ and 
\[ 
A\big(g_{\alpha}(s)\ast E_{\alpha}(-s^{\alpha}A)\big)(t)x=Ix-E_{\alpha}(-t^{\alpha}A)x.
\]
Moreover, in this case, we also get
\begin{equation}\label{nnn}
A\big(g_{\alpha}(s)\ast E_{\alpha}(-s^{\alpha}A)\big)(t)x=\big(g_{\alpha}(s)\ast AE_{\alpha}(-s^{\alpha}A)\big)(t)x,
\end{equation}
such that 
\[
\|A\big(g_{\alpha}(s)\ast E_{\alpha}(-s^{\alpha}A)\big)(t)x\|\leqslant C\|Ax\|t^{-\alpha\gamma}.
\]
\end{lemma}
\begin{proof}
Since $A$ is a closed operator, Lemmas \ref{i} and \ref{laplace}, and $(s^{\alpha}I+A)(s^{\alpha}I+A)^{-1}=I$ for $s^{\alpha}\in\rho(-A)$, we have
\begin{align*}
A\widehat{g_{\alpha}}(s)\widehat{E_{\alpha}(-t^{\alpha}A)}(s)x&=A\frac{1}{s^{\alpha}}\frac{s^{\alpha-1}}{s^{\alpha}I+A}x \\
&=\frac{1}{s}\frac{A}{s^{\alpha}I+A}x=\frac{1}{s}Ix-\frac{s^{\alpha-1}}{s^{\alpha}I+A}x=\widehat{1}(s)Ix-\widehat{E_{\alpha}(-t^{\alpha}A)}(s)x.
\end{align*}
So, the result follows by taking the inverse Laplace transform and Lemma \ref{laplace}.

Note now that by Lemma \ref{i} and the condition $1>\alpha(1+\gamma)$ we obtain
\begin{align*}
\|\big(g_{\alpha}(s)\ast AE_{\alpha}(-s^{\alpha}A)\big)(t)x\|&\leqslant C\|Ax\|\int_0^t (t-s)^{\alpha-1}s^{-\alpha(1+\gamma)}{\rm d}s\\
&\leqslant C\|Ax\|t^{-\alpha\gamma}\int_0^1 (1-r)^{\alpha-1}r^{-\alpha(1+\gamma)}{\rm d}r<+\infty.
\end{align*}
Hence, by \cite[Prop. 1.1.7]{vv}, the relation \eqref{nnn} holds.
\end{proof}
Next, we present the main results of this section.
\begin{theorem}\label{thm-linear-we}
Let $A\in \Theta_\omega^\gamma$ with $-1<\gamma<0$ and $\omega<\theta<\mu<\pi-\alpha\frac{\pi}{2}.$ Assume that $f(t)\in D(A)$ for any $t\in(0,T]$, $f\in L^1\big((0,T);D(A)\big)$ and let $f$ be H\"older continuous with an exponent $\nu\in(0,1]$ such that $\nu>\alpha(1+\gamma).$ Then 
\[
w(t)=(g_{\alpha-1}(s)\ast E_{\alpha}(-s^{\alpha}A)\ast f(s))(t)
\]
is the unique classical solution of \eqref{we-f} with $w_0=w_1=0.$
\end{theorem}
\begin{proof}
If $f\in L^1\big((0,T);D(A)\big)$ and $g_{\alpha-1}\in L^1(0,T)$ then $g_{\alpha-1}\ast f\in L^1\big((0,T);D(A)\big),$ see \cite[Prop. 1.3.1]{vv}. Since $E_{\alpha}(-t^{\alpha}A)$ is strongly continuous for any $x\in D(A)$ (Theorem \ref{strong}), by \cite[Prop. 1.3.4]{vv}, it follows that $w$ exits and defines a continuous function, that is $w(t)=(g_{\alpha-1}(s)\ast E_{\alpha}(-s^{\alpha}A)\ast f(s))(t)=(E_{\alpha}(-s^{\alpha}A)\ast g_{\alpha-1}(s)\ast f(s))(t)$ belongs to $C([0,T];X).$  

In some cases, we use the following equality:
\[
w(t)=\big(g_{\alpha-1}(s)\ast E_{\alpha}(-s^{\alpha}A)\ast f(s)\big)(t)=\int_0^t (t-s)^{\alpha-1}E_{\alpha,\alpha}(-(t-s)^{\alpha}A)f(s){\rm d}s.
\]
Let us prove that $w(t)\in D(A)$ for any $t\in(0,T].$ First, we write $w(t)=u_1(t)+u_2(t),$ where 
\[
u_1(t)=\int_0^t (t-s)^{\alpha-1}E_{\alpha,\alpha}(-(t-s)^{\alpha}A)[f(s)-f(t)]{\rm d}s,\quad 0<t\leqslant T,
\]
and
\[
u_2(t)=\int_0^t (t-s)^{\alpha-1}E_{\alpha,\alpha}(-(t-s)^{\alpha}A)f(t){\rm d}s, \quad 0<t\leqslant T.
\]
Note that $u_2(t)=\big(g_{\alpha-1}(s)*E_\alpha(-s^{\alpha}A)*1\big)(t)f(t)=\big(E_\alpha(-s^{\alpha}A)*g_{\alpha}(s)\big)(t)f(t).$ Since $f(t)\in D(A)$ and $1>\alpha(1+\gamma),$ by Lemma \ref{uno}, $u_2 \in D(A)$ and
\begin{equation}\label{555}
Au_2(t)=f(t)-E_{\alpha}(-t^{\alpha}A)f(t),\quad 0<t\leqslant T.    
\end{equation}

Now, from Lemma \ref{777}, the condition $\nu>\alpha(1+\gamma)$ and the H\"older continuity of $f$, we have
\begin{align*}
\|A(g_{\alpha-1}\ast E_{\alpha})(t-s)(f(s)-f(t))\|\leqslant C(t-s)^{-1-\alpha(1+\nu)}(t-s)^{\nu}\in L^1(0,t).    
\end{align*}
Hence, by \cite[Prop. 1.1.7]{vv}, $u_1\in D(A),$ and then $w\in D(A).$

We need to prove that $\partial_t^{\alpha}w\in C((0,T];X).$ For this, we have to show that $Aw\in C((0,T];X)$ and $^{C}\partial_{t}^{\alpha}w(t)=Aw(t)+f(t).$

Let $v(t)=Aw(t)+f(t).$ By Theorem \ref{derivative-s}, we obtain
\begin{align*}
    v(t)=A(g_{\alpha-1}(s)\ast E_{\alpha}(-s^{\alpha}A)\ast f)(t)+f(t)=-\big(E^{\prime}_{\alpha}(-s^{\alpha}A)\ast f\big)(t)+f(t).
\end{align*}

If $Aw\in C((0,T];X),$ we know that $v\in C((0,T];X).$ So, by Fubini's theorem, Theorem \ref{strong} with $f(t)\in D(A)$ $(t\in(0,T])$, it follows that 

\[
\int_0^t v(s){\rm d}s=(E_{\alpha}(-s^{\alpha}A)\ast f)(t),
\]
and hence 
\begin{equation}\label{v0}
v(t)=\frac{\partial}{\partial t}(E_{\alpha}(-t^{\alpha}A)\ast f)(t).
\end{equation}
Clearly $(E_{\alpha}(-s^{\alpha}A)\ast f)(t)\in C^1((0,T];X)).$ Thus, by the semigroup property of the Riemann Liouville fractional integral and \cite[Formula (1.21)]{thesis2001}, we get
\begin{align*}
^{C}\partial_t^{\alpha}w(t)&=\,^{C}\partial_t^{\alpha}\big(g_{\alpha-1}(s)\ast E_{\alpha}(-s^{\alpha}A)\ast f(s)\big)(t) \\
&=\,^{C}\partial_t^{\alpha}\,^{RL}I^{\alpha-1} \big(E_{\alpha}(-s^{\alpha}A)\ast f(s)\big)(t)\\
&=\,^{C}\partial_t^{\alpha}\,^{RL}I^{\alpha-1}\,^{RL}I^{1}v(t)=\,^{C}\partial_t^{\alpha}\,^{RL}I^{\alpha}v(t)=v(t)=-Aw(t)+f(t).
\end{align*}
It then remains to prove that $Aw\in C((0,T);X).$

By \eqref{555} and Theorem \ref{strong}, we have that $Au_2(t)$ is continuous on $(0,T].$

Now, let $h>0$ and $t\in(0,T],$ and write $Au_1(t+h)-Au_1(t)=h_1+h_2+h_3$ where
\[
h_1=\int_0^{t}A \big((g_{\alpha-1}(s)\ast E_{\alpha}(-s^{\alpha}A))(t+h-s)-(g_{\alpha-1}(s)\ast E_{\alpha}(-s^{\alpha}A))(t-s)\big)\big(f(s)-f(t)\big){\rm d}s,
\]
\[
h_2=\int_0^{t}A \big(g_{\alpha-1}(s)\ast E_{\alpha}(-s^{\alpha}A))\big)(t+h-s)\big(f(t)-f(t+h)\big){\rm d}s,
\]
and 
\[
h_3=\int_t^{t+h}A\big(g_{\alpha-1}(s)\ast E_{\alpha}(-s^{\alpha}A))\big)(t+h-s)\big(f(s)-f(t+h)\big){\rm d}s.
\]
From Lemma \ref{777} and $\nu>\alpha(1+\gamma)$, we obtain 
\begin{align*}
    \|h_2\|&\leqslant \int_0^{t}\|A \big(g_{\alpha-1}(s)\ast E_{\alpha}(-s^{\alpha}A))\big)(t+h-s)\|\|f(t)-f(t+h)\|{\rm d}s \\
&\leqslant Ch^{\nu}\int_0^{t}(t+h-s)^{-1-\alpha(1+\gamma)}{\rm d}s\leqslant Ch^{\nu-\alpha(1+\gamma)}\to 0\quad\text{as}\quad h\to0.
\end{align*}
Also, since $\nu>\alpha(1+\gamma),$ it follows that
\begin{align*}
    \|h_3\|&\leqslant C\int_t^{t+h}(t+h-s)^{-1-\alpha(1+\gamma)+\nu}{\rm d}s\leqslant Ch^{-\alpha(1+\gamma)+\nu}\to0\quad\text{as}\quad h\to0.
\end{align*}
Note that
\begin{align*}
\lim_{h\to0}A \big((g_{\alpha-1}(s)\ast E_{\alpha}(-s^{\alpha}A))&(t+h-s)\big(f(s)-f(t)\big) \\
&=A(g_{\alpha-1}(s)\ast E_{\alpha}(-s^{\alpha}A))(t-s)\big)\big(f(s)-f(t)\big).
\end{align*}
Also, by Lemma \eqref{777}, it yields that
\begin{align*}
\|A\big((g_{\alpha-1}(s)&\ast E_{\alpha}(-s^{\alpha}A))(t+h-s)\big(f(s)-f(t)\big)\| \\
&\leqslant C(t+h-s)^{-1-\alpha(1+\gamma)}(t-s)^{\nu}\leqslant C(t-s)^{-1-\alpha(1+\gamma)+\nu}\in L^1(0,t).
\end{align*}
Thus, by using the Dominated Convergence Theorem, we obtain that $h_1\to0$ as $h\to0.$
Then $Aw\in C((0,T];X).$

Let us check that $w(0)=w^{\prime}(0)=0.$ As $\big(E_{\alpha}(-s^{\alpha}A)\ast f(s)\big)(t)\in C([0,T];X)$ we arrive at 
\[
\|w(t)\|\leqslant C\|(E_{\alpha}(-s^{\alpha}A)\ast f(s))(t)\|_{C([0,T];X)}\int_0^t (t-s)^{\alpha-2}{\rm d}s\to 0\quad\text{as}\quad t\to0.
\]

So, $w(0)=0.$ It remains to check that $w^{\prime}(0)=0.$ Indeed, since $w^{\prime}(t)=\prescript{RL}{0}I_t^{\alpha-1}v(t),$ $^{C}\partial_t^{\alpha}w(t)=v(t),$ $v\in C((0,T];X)$ and \cite[Formula (1.21)]{thesis2001} then 
\[
w^{\prime}(t)=\prescript{RL}{0}I_t^{\alpha-1}\prescript{C}{}{\partial_t^{\alpha}w(t)}=\prescript{RL}{0}I_t^{\alpha-1}\prescript{C}{}{\partial_t^{\alpha-1}w^{\prime}(t)}=w^{\prime}(t)-w^{\prime}(0).
\]
Hence $w^{\prime}(0)=0.$ Finally, $w(t)=(g_{\alpha-1}(s)\ast E_{\alpha}(-s^{\alpha}A)\ast f(s))(t)$ is the unique classical solution of \eqref{we-f}. Note that the uniqueness of the solution follows by the Laplace transform, Lemma \ref{laplace}, $1>\alpha(1+\gamma)$ and the uniqueness of the inverse Laplace transform.
\end{proof}


\begin{remark}
    Theorem \ref{thm-linear-we} gives a restriction on a region, mainly, $\nu>\alpha(1+\gamma).$ This restriction also appears in the case where $\alpha\in(0,1),$ see \cite[Theorem 4.1]{section3}. Note that some initial conditions varied in this paper with respect to the results in \cite{section3}, but the restriction on the h\"older exponent remains the same. So, until now, what happens in the region $\nu\leqslant\alpha(1+\gamma)$ has not been known or claimed. Maybe, perhaps, we can gain more regularity in the mentioned results by changing the methods or tools. This remains an open question. Also, note that condition $f\in L^1\big((0,T);D(A)\big)$ is necessary to guarantee the continuity of $w$ in $C([0,T];X).$
\end{remark}

The following statement follows by Theorems \ref{h-cs-we} and \ref{thm-linear-we}.

\begin{theorem}\label{general}
Let $A\in \Theta_\omega^\gamma$ with $-1<\gamma<-1/2$ and $\omega<\theta<\mu<\pi-\alpha\frac{\pi}{2}.$ Also, suppose that $1<\alpha<2$ with $\frac{1}{1+\gamma}>\alpha>\frac{1}{-\gamma}.$ Assume that $f(t)\in D(A)$ for any $t\in(0,T]$, $f\in L^1\big((0,T);D(A)\big)$ and let $f$ be H\"older continuous with an exponent $\nu\in(0,1]$ such that $\nu>\alpha(1+\gamma).$ If $w_0,w_1\in D(A)$, then 
\[
w(t)=E_{\alpha}(-t^{\alpha}z)(A)w_0+tE_{\alpha,2}(-t^{\alpha}z)(A)w_1+(g_{\alpha-1}(s)\ast E_{\alpha}(-s^{\alpha}A)\ast f(s))(t)
\]
is the unique classical solution of \eqref{we-f}. 
\end{theorem}

\section{Semilinear abstract Volterra equations of wave type}\label{semilinear-s}
\begin{align}\label{we-f-semi}
\prescript{C}{}\partial_{t}^{\alpha}w(t)-Aw(t)&=f(t,w(t)),\quad t\in(0,T],\quad 1<\alpha<2,\nonumber \\
w(t)|_{_{_{t=0}}}&=w_0, \\
\partial_t w(t)|_{_{_{t=0}}}&=w_1,\nonumber
\end{align}
where $X$ is a complex Banach space, $A\in \Theta_\omega^\gamma$ with $\omega<\theta<\mu<\pi-\alpha\frac{\pi}{2}.$

\medskip The classical solution of \eqref{we-f-semi} will be analyzed from the regularity of the mild solution. Therefore, we first give such a definition. In this case, taking into account the analysis done in Theorems \ref{thm-linear-we} and \ref{general}, we arrive at the following definition.  

\begin{definition}\label{mild-s}
    Let $A\in \Theta_\omega^\gamma$ with $\omega<\theta<\mu<\pi-\alpha\frac{\pi}{2}.$ A function $w\in C([0,T];X)$ (respectively, $w\in C([0,T];D(A)))$ is called a mild solution of \eqref{we-f-semi} if $w$ satisfies \[
    w(t)=E_{\alpha}(-t^{\alpha}z)(A)w_0+(1\ast E_{\alpha}(-s^{\alpha}A))(t)w_1+(g_{\alpha-1}(s)\ast E_{\alpha}(-s^{\alpha}A)\ast f(s,w(s)))(t).
    \]
\end{definition}

We can arrive at the above definition by assuming that $w\in C([0,T];X)$ satisfies the equation \eqref{we-f-semi}. Indeed, assuming the existence of the vector Laplace transform for $w,$ we get in equation \eqref{we-f-semi} that
\[
\lambda^{\alpha}\widehat{w}(\lambda)-\lambda^{\alpha-1}w_0-\lambda^{\alpha-2}w_1-A\widehat{w}(\lambda)=\widehat{f(t,w(t))}(\lambda).
\]
If $\lambda^{\alpha}\in\rho(A),$ we obtain that
\[
\widehat{w}(\lambda)=\frac{\lambda^{\alpha-1}}{\lambda^{\alpha}+A}w_0+\frac{\lambda^{\alpha-2}}{{\lambda^{\alpha}+A}}w_1+\frac{1}{{\lambda^{\alpha}+A}}\widehat{f(t,w(t))}(\lambda).
\]
By Lemma \ref{laplace}, we get 
\[
\widehat{w}(\lambda)=\widehat{S_{\alpha}(t)}w_0+\widehat{1}(\lambda)\widehat{S_{\alpha}(t)}w_1+\widehat{g_{\alpha-1}(t)}(\lambda)\widehat{S_{\alpha}(t)}\widehat{f(t,w(t))}(\lambda).
\]
Taking the inverse Laplace transform and using some classical formulas, it yields the above representation given in Definition \ref{mild-s}. It is necessary to mention that in our case we usually assume the data $w_0,w_1$ in $D(A),$ and hence it follows the continuity in the closed interval, see Remark \ref{remark-main-thm}. Of course, we also need to impose some appropriate conditions over the nonlinear function $f(t,w(t))$ to guarantee the latter affirmation. It must be clear at this point that taking $w_0\in X$ and $w_1=0$ then we have $E_{\alpha}(-t^{\alpha}A)w_0\in C((0,T];X),$ and thus the mild solution belongs to $C((0,T];X);$ again see Remark \ref{remark-main-thm}. Also, from Theorem \ref{bounded}, Lemma \ref{laplace} and condition $1>\alpha(1+\gamma)$, we can guarantee the existence of the Laplace transform for $w.$   

\medskip Let us start by showing the result on the existence of a mild solution for \eqref{we-f-semi}.

\begin{theorem}\label{existence-mild}
Suppose that $A\in \Theta_\omega^\gamma$ with $\omega<\theta<\mu<\pi-\alpha\frac{\pi}{2},$ $-1<\gamma<0$ and $1>\alpha(1+\gamma).$ Assume that the nonlinear function $f(t,x):[0,T]\times X\to D(A)$ is continuous with respect to the time variable $t$ and there exists a constant $L>0$ such that
\begin{equation}\label{condition}
    \|f(t,x)-f(t,y)\|_{D(A)}\leqslant L\|x-y\|\quad\text{for any}\quad t\in[0,T]\quad\text{and}\quad x,y\in X.
\end{equation}
Then the problem \eqref{we-f-semi} has a unique mild solution in $C([0,T];D(A))$ for $w_0,w_1,Aw_0\in D(A).$
\end{theorem}
\begin{proof}
Take the Banach space $C([0,T];D(A))$ endowed with the norm 
\[
\|w\|_{C([0,T];D(A))}=\sup_{t\in[0,T]}\big(\|w(t)\|+\|Aw(t)\|\big).
\]
Consider the operator define by 
\begin{align}\label{operator}
(Hw)(t)=E_{\alpha}(-t^{\alpha}A)w_0+&tE_{\alpha,2}(-t^{\alpha}A)w_1 \\
&+(g_{\alpha-1}(s)\ast E_{\alpha}(-s^{\alpha}A)\ast f(s,w(s)))(t). \nonumber   
\end{align}
Let us see that $H:C([0,T];D(A))\to C([0,T];D(A)).$ From Theorems \ref{bounded} and \ref{strong}, $t_0\in[0,T]$ and the conditions $w_0,Aw_0\in D(A),$ we get
\[
\|E_{\alpha}(-t^{\alpha}A)w_0-E_{\alpha}(-t_0^{\alpha}A)w_0\|_{D(A)}\to0,\quad\text{as}\quad t\to t_0.
\]
This gives $E_{\alpha}(-t^{\alpha}A)w_0\in C([0,T];D(A)).$ Also, by Theorem \ref{bounded}, the condition $1>\alpha(1+\gamma)$ and for $t_0>0$, we have that
\begin{align*}
\|(1\ast E_{\alpha}(-s^{\alpha}A))(t)w_1&-(1\ast E_{\alpha}(-s^{\alpha}A))(t_0)w_1\| \\
&\leqslant C\|w_1\|\int_{t_0}^t s^{-\alpha(1+\gamma)}{\rm d}s\to0,\quad\text{as}\quad t\to t_0.
\end{align*}
Now, if $t_0\to 0$, we obtain
\[
\|(1\ast E_{\alpha}(-s^{\alpha}A))(t_0)w_1\|\leqslant C\|w_1\|\int_0^{t_0} s^{-\alpha(1+\gamma)}{\rm d}s\to0.
\]
Thus, for any $t_0\in[0,T],$ it follows that 
\[
\|(1\ast E_{\alpha}(-s^{\alpha}A))(t)w_1-(1\ast E_{\alpha}(-s^{\alpha}A))(t_0)w_1\|\to0,\quad\text{as}\quad t\to t_0.
\]
Here, by Lemma \ref{i}, for $t_0>0,$ we also get
\begin{align*}
\|A\big[(1\ast E_{\alpha}(-s^{\alpha}A))(t)w_1&-(1\ast E_{\alpha}(-s^{\alpha}A))(t_0)w_1\big]\| \\
&\leqslant C\|Aw_1\|\int_{t_0}^t s^{-\alpha(1+\gamma)}{\rm d}s\to0,\quad\text{as}\quad t\to t_0.
\end{align*}
Now, if $t_0\to0,$ from the same Lemma \ref{i}, we arrive at
\[
\|A(1\ast E_{\alpha}(-s^{\alpha}A))(t_0)w_1\|\leqslant C\|Aw_1\|\int_0^{t_0} s^{-\alpha(1+\gamma)}{\rm d}s\to0.
\]
This implies that $(1\ast E_{\alpha}(-s^{\alpha}A))(t)w_1\in C([0,T];D(A)).$

\medskip Note also that $f(t,w(t))\in C([0,T];D(A))$ for any $w\in C([0,T];D(A)).$ In fact, for $w\in C([0,T];D(A))$ and $t_0\in [0,T],$ we obtain
\begin{align*}
\|f(t,w(t))&-f(t_0,w(t_0))\|_{D(A)} \\
&\leqslant \|f(t,w(t))-f(t,w(t_0))\|_{D(A)}+\|f(t,w(t_0))-f(t_0,w(t_0))\|_{D(A)}.
\end{align*}
Hence, by \eqref{condition}, $w\in C([0,T];D(A))$ and taking $t\to t_0$, we have
\[
\|f(t,w(t))-f(t,w(t_0))\|_{D(A)}\leqslant L\|w(t)-w(t_0)\|\leqslant L\|w(t)-w(t_0)\|_{D(A)}\to0.
\]
Now, by the continuity in time of $f,$ it follows that
\[
\|f(t,w(t_0))-f(t_0,w(t_0))\|_{D(A)}\to0,\quad\text{as}\quad t\to t_0.
\]
Therefore, $\|f(t,w(t))-f(t_0,w(t_0))\|_{D(A)}\to0$ when $t\to t_0,$ and it is equivalent to $f(\cdot,w(\cdot))\in C([0,T];D(A)).$ By \cite[Prop. 1.3.1]{vv} and $g_{\alpha-1}\in L^1[0,T]$ we have $(g_{\alpha-1}(s)\ast f(s,w(s)))(t)\in L^1\big([0,T];D(A)\big).$ From the fact that $E_{\alpha}(-t^{\alpha}A)$ is strongly continuous for any $x\in D(A)$ (Theorem \ref{strong}), and by \cite[Prop. 1.3.4]{vv}, it follows that $(g_{\alpha-1}(s)\ast E_{\alpha}(-s^{\alpha}A)\ast f(s,w(s)))(t)$ exits and defines a continuous function, that is, the function belongs to $C([0,T];D(A)).$ Therefore, the operator $H$ is well defined.

\medskip On the other hand, suppose that $w,v\in C([0,T];D(A)).$ By Lemma \ref{777}, we have 
\[
\|(g_{\alpha-1}(s)\ast E_{\alpha}(-s^{\alpha}A))(t)\|\leqslant Ct^{-1-\alpha\gamma}.
\]
Hence
\begin{align*}
    \|(Hw)(t)&-(Hv)(t)\| \\
    &\leqslant C\int_0^t \|g_{\alpha-1}(s)\ast E_{\alpha}(-s^{\alpha}A)(t-s)\|\|f(s,w(s))-f(s,v(s))\|{\rm d}s \\
    &\leqslant C\int_0^t (t-s)^{-1-\alpha\gamma}\|f(s,w(s))-f(s,v(s))\|{\rm d}s.
\end{align*}
Also, from Theorems \ref{thm-main} and \ref{lem-a1}, and Lemma \ref{777}, we get that 
\begin{align*}
\|A\big((Hw)(t)&-(Hv)(t)\big)\| \\
&\leqslant C\int_0^t \|g_{\alpha-1}(s)\ast E_{\alpha}(-s^{\alpha}A)(t-s)\|\|A\big(f(s,w(s))-f(s,v(s))\big)\|{\rm d}s \\
    &\leqslant C\int_0^t (t-s)^{-1-\alpha\gamma}\|A\big(f(s,w(s))-f(s,v(s))\big)\|{\rm d}s.
\end{align*}
Thus, by the above estimates and \eqref{condition}, it yields 
\begin{align*}
\|(Hw)(t)&-(Hv)(t)\|_{D(A)}\leqslant C\int_0^t (t-s)^{-1-\alpha\gamma}\|f(s,w(s))-f(s,v(s))\|_{D(A)}{\rm d}s \\
&\leqslant CL\int_0^t (t-s)^{-1-\alpha\gamma}\|w(s)-v(s)\|_{D(A)}{\rm d}s \\
&\leqslant \frac{CLt^{-\alpha\gamma}}{-\alpha\gamma}\|w-v\|_{C([0,T];D(A))}\leqslant \frac{CLT^{-\alpha\gamma}}{-\alpha\gamma}\|w-v\|_{C([0,T];D(A))}.
\end{align*}
By mathematical induction, we obtain the following inequality:
\[
\|(H^{n}w)(t)-(H^{n}v)(t)\|_{D(A)}\leqslant \frac{(CLT^{-\alpha\gamma})^n}{(-\alpha\gamma)^n n!}\|w-v\|_{C([0,T];D(A))}.
\]
Since $\displaystyle\lim_{n\to+\infty}\frac{(CLT^{-\alpha\gamma})^n}{(-\alpha\gamma)^n n!}=0,$ we see that $H^n$ is a contraction map and therefore has a unique fixed point. This completes the proof.
\end{proof}

\begin{remark}
    There are two interesting observations to highlight in Theorem \ref{existence-mild}. First, for the existence of the mild solution of problem \eqref{we-f-semi}, an additional condition, never before requested in this paper, has been imposed, that is, $Aw_0\in D(A).$ This condition is necessary to prove that $E_{\alpha}(-t^{\alpha}A)\in D(A)$ for any $t\in [0,T].$ Note that in Theorem \ref{h-cs-we}, to prove the classical solutions of equation \eqref{we}, it was not necessary to make such a condition. Basically, from the definition of classical solutions we guarantee $E_{\alpha}(-t^{\alpha}A)\in D(A)$ for any $t\in (0,T].$ Thus, it is clear that the extra condition gives the belonging over the domain of $A$ in $t=0,$ which is a delicate step with this type of propagators. Second, for the propagator $tE_{\alpha,2}(-t^{\alpha}A)$ is not necessary imposed that $Aw_1\in D(A).$ However, we use the alternative representation for it (Remark \ref{rerewrite}), and for some estimates we need to restrict to the following region $1>\alpha(1+\gamma).$ This way allows us to actively take advantage of the strongly continuity of the propagator $E_{\alpha}(-t^{\alpha}A)$, see Theorem \ref{strong}.    
\end{remark}

Next, we provide the main result of this section on the regularity of the mild solution. Here, we show that, under a Lipchitz-type condition over $f$, it becomes a classical solution. It is important to note that our nonlinear function must be properly defined from $[0,T]\times X$ to $D(A).$ The restriction on the range comes from previous results (specifically Theorem \ref{general}) and at this point we cannot escape from it.  

\begin{theorem}\label{ult}
Let $A\in \Theta_\omega^\gamma$ with $-1<\gamma<-1/2$ and $\omega<\theta<\mu<\pi-\alpha\frac{\pi}{2}.$ Also, suppose that $1<\alpha<2$ with $\frac{1}{1+\gamma}>\alpha>\frac{1}{-\gamma}.$ Assume that for any $k>0,$ there exits a constant $L(k)$ such that the function $f:[0,T]\times X\to D(A)$ satisfies 
\begin{equation}\label{condition-final}
        \|f(t,w)-f(s,v)\|_{D(A)}\leqslant L(k)\big(|t-s|^{\nu}+\|w-v\|\big),\quad\text{for some}\quad \nu>\alpha(1+\gamma),
    \end{equation}
    for any $t,s\in [0,T],$ $w,v\in X$ with $\|w\|,\|v\|\leqslant k.$ If $w_0,w_1\in D(A)$ and $w\in C([0,T];X)$ is a mild solution of \eqref{we-f-semi}, then $w$ is a classical solution of \eqref{we-f-semi}.
\end{theorem}
\begin{proof}
The idea of the proof is applied the Theorem \ref{general}. Thus, we need to show that the function $f$ satisfies the H\"older condition imposed in such a statement. First, for $h>0$ and $t\in [0,T-h],$ we have
\begin{align*}
w(t+h)&-w(t)=E_{\alpha}(-(t+h)^{\alpha}A)w_0-E_{\alpha}(-t^{\alpha}A)w_0 \\
&\,+(1\ast E_{\alpha}(-s^{\alpha}A))(t+h)w_1-(1\ast E_{\alpha}(-s^{\alpha}A))(t)w_1 \\
&\,+\int_0^h (g_{\alpha-1}(r)\ast E_{\alpha}(-r^{\alpha}A))(t+h-s)f(s,w(s)){\rm d}s \\
&\,+\int_0^t (g_{\alpha-1}(r)\ast E_{\alpha}(-r^{\alpha}A))(t-s)\big[f(s+h,w(s+h))-f(s,w(s)))\big]{\rm d}s.
\end{align*}
Note that by Remark \ref{conmu} and $w_0\in D(A)$ we get
\begin{align*}
    E_{\alpha}(-(t+h)^{\alpha}A)w_0&-E_{\alpha}(-t^{\alpha}A)w_0 
 \\
&=\int_t^{t+h}\partial_s E_{\alpha}(-s^{\alpha}A)w_0{\rm d}s=-\int_t^{t+h}s^{\alpha-1}AE_{\alpha,\alpha}(-s^{\alpha}A)w_0{\rm d}s.
\end{align*}
From Theorem \ref{bounded} and the two-sided inequality $(a+b)^{\rho}\asymp a^{\rho}+b^{\rho},$ for $a,b,\rho\geqslant0$ ($C_1 [a^{\rho}+b^{\rho}]\leqslant (a+b)^{\rho}\leqslant C_2 [a^{\rho}+b^{\rho}]$ for some constants $C_{1,2}>0$), we have 
\begin{align*}
\|E_{\alpha}(&-(t+h)^{\alpha}A)w_0-E_{\alpha}(-t^{\alpha}A)w_0\| \\
&\leqslant C\|Aw_0\|\int_t^{t+h}s^{\alpha-1-\alpha(1+\gamma)}{\rm d}s\leqslant C\|Aw_0\|\big((t+h)^{-\alpha\gamma}-t^{-\alpha\gamma}\big)\leqslant C\|Aw_0\|h^{-\alpha\gamma}.
\end{align*}
Now, from \eqref{condition-final}, it follows that
\[
\|f(s+h,w(s+h))-f(s,w(s))\|\leqslant L(k)\big(h^{\nu}+\|w(s+h)-w(s)\|\big).
\]
Also, from Lemma \ref{777}, we obtain
\begin{align*}
    &\int_0^t \|(g_{\alpha-1}(r)\ast E_{\alpha}(-r^{\alpha}A))(t-s)\|\|f(s+h,w(s+h))-f(s,w(s)))\|{\rm d}s \\
&\leqslant C\int_0^t (t-s)^{-1-\alpha\gamma}L(k)\big(h^{\nu}+\|w(s+h)-w(s)\|\big){\rm d}s \\
&\leqslant C\left(h^{\nu}T^{-\alpha\gamma}+\int_0^t (t-s)^{-1-\alpha\gamma}\|w(s+h)-w(s)\|{\rm d}s\right).
\end{align*}
Again, by Lemma \ref{777}, we also have
\begin{align*}
\int_0^h &\|(g_{\alpha-1}(r)\ast E_{\alpha}(-r^{\alpha}A))(t+h-s)\|\|f(s,w(s))\|{\rm d}s \\
&\leqslant C\sup_{s\in[0,T]}\|f(s,w(s))\|\int_0^h (t+h-s)^{-1-\alpha\gamma}{\rm d}s \\
&\leqslant C\sup_{s\in[0,T]}\|f(s,w(s))\|\big((t+h)^{-\alpha\gamma}-t^{\alpha\gamma}\big)\leqslant Ch^{-\alpha\gamma}\sup_{s\in[0,T]}\|f(s,w(s))\|.
\end{align*}
Besides, from Theorem \ref{bounded}, we get
\begin{align*}
\|(1\ast E_{\alpha}(-r^{\alpha}A))&(t+h)w_1-(1\ast E_{\alpha}(-r^{\alpha}A))(t)w_1\|\leqslant C\|w_1\|\int_{t}^{t+h}s^{-\alpha(1+\gamma)}{\rm d}s \\
&\leqslant C\|w_1\|\big((t+h)^{1-\alpha(1+\gamma)}-t^{1-\alpha(1+\gamma)}\big)\leqslant C\|w_1\|h^{1-\alpha(1+\gamma)}.
\end{align*}
Finally
\begin{align*}
\|w(t+h)-w(t)\|&\leqslant C \big(\|Aw_0\|h^{-\alpha\gamma}+\|w_1\|h^{1-\alpha(1+\gamma)}+h^{-\alpha\gamma}\sup_{s\in[0,T]}\|f(s,w(s))\|\big) \\
&+C\left(h^{\nu}T^{-\alpha\gamma}+\int_0^t (t-s)^{-1-\alpha\gamma}\|w(s+h)-w(s)\|{\rm d}s\right) \\
&\leqslant C\left(h^{\theta}+\int_0^t (t-s)^{-1-\alpha\gamma}\|w(s+h)-w(s)\|{\rm d}s\right),
\end{align*}
where $\theta=\min\{\nu,1-\alpha(1+\gamma)\}>\alpha(1+\gamma).$ Here, note that $-\alpha\gamma>1-\alpha(1+\gamma).$ By a Grönwall type inequality \cite[Lemma 7.1.1]{[5]}, we obtain that 
\[
\|w(t+h)-w(t)\|\leqslant Ch^{\theta}.
\]
Hence, by \eqref{condition-final}, the function $f\big(\cdot,w(\cdot)\big):[0,T]\to D(A)$ satisfies 
\[
\|f(t+h,w)-f(t,w)\|_{D(A)}\leqslant L(k)\big(h^{\nu}+\|w(t+h)-w(t)\|\big)\leqslant CL(k)h^{\theta}.
\]
So, $f\big(\cdot,w(\cdot)\big)\in L^1\big((0,T);D(A)\big),$ and the result follows by Theorem \ref{general}.
\end{proof}

\begin{remark}
In Theorem \ref{ult}, we assume that $w\in C([0,T];X)$ is a mild solution of \eqref{we-f-semi}; so, in principle, we need to at least guarantee the existence of one. In fact, by Theorem \ref{existence-mild}, we just have to request that condition \eqref{condition} holds instead of \eqref{condition-final}, $f$ continuous with respect to the time variable $t$ and also suppose $Aw_0\in D(A).$ Of course, the latter result will be a direct consequence of Theorem \ref{ult}. Thus, we provide some conditions to show the existence of some mild solutions without claiming optimality or minimality.    
\end{remark}

\section*{\small
 Conflict of interest} 

 {\small
 The author declares that there are no conflicts of interest.}

\end{document}